\documentclass{article}
\usepackage{graphicx} % Required for inserting images
\usepackage{amsmath,amssymb}
\usepackage{amsthm}
\usepackage{fullpage}
\usepackage{mathtools}
\usepackage{thmtools}
\usepackage{hyperref}
\usepackage[dvipsnames]{xcolor}
\usepackage[normalem]{ulem} %strike out text + \emph als underline

\usepackage[english]{babel}
\usepackage[utf8]{inputenc}
\usepackage[T1]{fontenc}
\usepackage{lmodern} \normalfont %to load T1lmr.fd
\DeclareFontShape{T1}{lmr}{bx}{sc} { <-> ssub * cmr/bx/sc }{}
\usepackage{microtype}	% for improved spacing between words and letters

\usepackage{accents} % correct spacing for accents in sub- and superscript

\usepackage{tikz}
\usetikzlibrary{calc,positioning,shapes}
\usetikzlibrary{patterns,decorations.pathmorphing,decorations.markings}
\usetikzlibrary{arrows.meta,bending} % Preserve curved-arrow rendering formerly supplied by circuitikz.
\usepackage[margin=10pt,font=small,labelfont=bf,labelsep=endash]{caption}

\numberwithin{equation}{section}

\usepackage[noadjust]{cite}
\usepackage{enumitem} % for labelling items in enumerate
\setlist[enumerate]{label=(\roman*)}

\theoremstyle{plain}
\newtheorem{theorem}{Theorem}[section]

\newtheorem{proposition}[theorem]{Proposition}
\newtheorem{lemma}[theorem]{Lemma}
\newtheorem{corollary}[theorem]{Corollary}
\newtheorem{definition}[theorem]{Definition}
\newtheorem{example}[theorem]{Example}
\theoremstyle{definition}
\newtheorem{remark}[theorem]{Remark}

\definecolor{myBlue}{RGB}{30,144,255} % dodger blue
\definecolor{myGreen}{RGB}{69,169,0} % chatreuse
\definecolor{amethyst}{rgb}{0.6, 0.4, 0.8}

\def\N{\mathbb{N}}
\def\R{\mathbb{R}}
\def\C{\mathbb{C}}
\def\SS{\mathbb{S}}

\newcommand{\cO}{\mathcal{O}}
\newcommand{\abso}[1]{\lvert #1 \rvert}

\DeclareMathOperator{\sspan}{span}
\DeclareMathOperator{\diag}{diag}
\newcommand{\norm}[1]{\lVert#1\rVert}
\newcommand{\coleq}{\coloneqq}
\DeclareMathOperator{\Span}{span}

\newcommand{\calO}{\ensuremath{\mathcal{O}} }

\usepackage{verbatim}

\usepackage{float} % Provides the [H] figure placement option

\newcommand{\mHC}{m_{HC}}

\newcommand{\D}{\mathbb{D}}
\newcommand{\HHH}{\mathbb{H}}
\newcommand{\HH}{\mathcal{H}}
\newcommand{\B}{\mathcal{B}}

\DeclareMathOperator{\im}{im}
\DeclareMathOperator{\Id}{I} %{Id}

\title{Connection of hypocoercivity and hypocontractivity via the Cayley transform}

\author{
A.~Arnold\thanks{TU Wien, Institute of Analysis and Scientific Computing, Wiedner Hauptstr. 8-10, A-1040 Wien, Austria, anton.arnold@tuwien.ac.at},
S.~Egger\thanks{TU Wien, Institute of Analysis and Scientific Computing, Wiedner Hauptstr. 8-10, A-1040 Wien, Austria, stefan.egger@tuwien.ac.at} ,
V.~Mehrmann\thanks{{Technische Universität Berlin, Institut f.~Mathematik, IMoS/MA 4--5, Fasanenstr.~89, D--10623 Berlin, mehrmann@math.tu-berlin.de}},
and E.A.~Nigsch\thanks{TU Wien, Institute of Analysis and Scientific Computing, Wiedner Hauptstr. 8-10, A-1040 Wien, Austria, eduard.nigsch@tuwien.ac.at} \thanks{corresponding author}
}
\begin{document}

\maketitle

\begin{abstract}
The concepts of hypocoercivity and hypocontractivity and their relationship are studied for semi-dissipative continuous-time and discrete-time evolution equations in a Hilbert space setting. New proofs for the characterization of the short-time decay of the solution from the initial value are presented, that in particular characterize the constants in the leading terms of the solution when expanded in time. It is shown that the hypocoercivity index is equal to the hypocontractivity index when the implicit midpoint rule (the Cayley transformation) is used for the discretization. Maximally coercive/contractive representations of hypocoercive and hypocontractive semi-dissipative systems are presented, as well as the effect of different representations on the error estimates for the numerical solution.
\end{abstract}



\section{Introduction}
We are interested in applying time-discretization methods to linear time-invariant initial value problems,
\begin{equation}\label{linodecc}
    \dot x =%A_cx=
    -B_cx,\quad x(0)=x^0, \quad t>0,
\end{equation}
where $B_c$ is either a complex matrix or a bounded linear operator on an infinite dimensional, separable Hilbert space $\HH$. We are particularly interested in \emph{(semi-)dissipative systems}, i.e. systems where $B_H$, the Hermitian part of $B_c$, is positive
(semi-)definite. For simplicity of the presentation we  always assume here that $B_c$ has a trivial kernel. Otherwise, the space should be restricted to $(\ker B_c)^\perp$, assuming that $-B_c$ is semi-dissipative (see also \S{I} of \cite{Vil09}).

Analogously, we study discrete-time systems (that may or may not arise from a time-discretization of \eqref{linodecc}) which have the form
\begin{equation}\label{linodedd}
    x_{k+1} =B_dx_k,\quad x_0=x^0, \quad k\in\N_0,
\end{equation}
where $B_d$ is a bounded linear operator on $\HH$.
We are particularly interested in systems that are semi-contractive: We call such a system \emph{(semi-)contractive} if $\sigma_{\max}(B_d)\, (\leq) < 1$. Here, $\sigma_{\max}(B_d)$ is the largest singular value (the \emph{spectral norm}) of $B_d$. In analogy to the continuous-time case, we  always assume that $B_d$ does not have the eigenvalue $1$. Otherwise, the space should be restricted to $(\ker (B_d - I))^\perp$, assuming that $B_d$ is semi-contractive. A detailed study of the properties of such systems and their relationship in the finite dimensional case has been presented in \cite{AchAM23ELA}.

In this paper new proofs for the characterization of the short-time decay of the solutions  from
the initial value are presented for both systems. It is shown that essentially the same proof technique can be employed in both systems. The constants in the leading terms of the solution are determined when it is expanded in time in the neighborhood of the initial value. When the discrete-time system arises from the implicit midpoint discretization (scaled Cayley-transform, Crank-Nicolson scheme) of a continuous-time system, it is shown that the norm of the continuous-time solution operator is approximated to a higher order than expected from the order of the discretization.

Since the representation of a linear system as a semi-dissipative or semi-contractive system is not unique, but can be modified by a change of basis, we discuss the construction of maximally coercive/contractive representations of hypocoercive and hypocontractive systems and the effect of different representations
on the error estimates for the numerical solution.

The paper is organized as follows. In Section~\ref{sec:main} we  present the notation and recall the previously derived decay results for continuous-time systems.
Section~\ref{sec:short-time} presents a new proof of the short-time decay in the continuous-time case. The relationship between the hypocoercivity index of a continuous-time system and the hypocontractivity index
of the system discretized in time via the implicit midpoint rule, which leads to the scaled Cayley transform of the system operator, is studied in Section~\ref{sec:Cayley}. Section~\ref{sec:t-discrete} studies the concept of hypocontractivity, and the corresponding hypocontractivity index, as finite difference approximation of the corresponding hypocoercivity concept and index.
Maximally coercive/contractive representations via a change of basis of hypocoercive, respectively hypocontractive systems, as coercive respectively contractive systems are discussed in Section~\ref{sec:optimal} as well as the effect of these transformations on the short-time decay.

%%%%%%%%%%%%%%%%%%%%%%%%%%%%%%%%%%%%%%%%%%%%%%%%%%%%%%

\section{Notation and preliminaries}\label{sec:main}

In this section we introduce the notation and recall the previous initial-time decay results.

For  a separable Hilbert space $\HH$, by
$\mathcal B(\HH)$ we denote the set of bounded linear operators on  $\HH$. For an operator $B \in \mathcal{B}(\HH)$, its spectrum is denoted by $\Lambda(B)$, its spectral radius by $\rho(B)$ and its set of singular values by $\sigma(B)$. Moreover, $B^*$ denotes the adjoint of $B$. The {unique} positive {semi-}definite square root of a self-adjoint positive semidefinite operator $R$ is denoted as $\sqrt{R}$.
For the matrix case $\|\cdot\|$ denotes the spectral norm of vectors and matrices, otherwise the Hilbert space norm.
We make frequent use of the sets $\R^+ \coloneqq (0,\infty)$, $\SS:=\{x\in\HH\,:\, \|x\|=1\}$,
$\HHH \coloneqq \{z \in \C: \Re(z) < 0\}$, and $\D \coloneqq \{z \in \C: \lvert z \rvert < 1 \}$.

For scalar-valued $g$ with $g(t)\ne0$ for all sufficiently small positive $t$, we write $f(t) = o(g(t))$ for $t \to 0+$ if $\|f(t)/g(t)\| \to 0$ for $t \to 0+$. Similarly, $f(t) = \Theta(g(t))$ for $t \to 0+$ means that $0 < \liminf_{t\to 0+} \| f(t)/g(t) \| \leq \limsup_{t\to 0+} \| f(t)/g(t) \| < \infty$. Note that these estimates do not include the point $t=0$.

The concepts of semi-dissipativity and hypocoercivity for operators in $\mathcal B(\HH)$ are defined as follows, see \cite{AchAM25,AchAMN25,Vil09}.
\begin{definition}\label{def:HC+semidiss}
Let $B_c=B_H-B_S
\in \mathcal B(\HH)$ with $B_H \coloneqq \frac{B_c + B_c^*}{2}$ and
$B_S \coloneqq -\frac{B_c - B_c^*}{2}$.
\begin{enumerate}
\item $-B_c$ is called \emph{semi-dissipative} (or, equivalently, $B_c$ is called \emph{semi-accretive}) if $B_H \geq 0$, i.e., $B_H$ is positive semidefinite.
\item $B_c$ is called \emph{hypocoercive} if there exist constants $C\ge1$ and $\lambda>0$ such that the solutions to \eqref{linodecc} satisfy
\begin{equation}\label{def:HC}
  \forall x^0\in\HH,\quad \forall t\ge0:\quad
  \|x(t)\| \le C e^{-\lambda t}\|x^0\|.
\end{equation}
\item Let $-B_c$ be semi-dissipative. The \emph{hypocoercivity index} (HC-index) $m_{HC}$ of $B_c$ is defined as the smallest integer $m\in\N_0$ (if it exists) such that
\begin{equation}\label{def:mHC}
    \sum_{j=0}^m (B_c^*)^jB_H \,B_c^j \ge \kappa I
\end{equation}
for some $\kappa>0$.
\end{enumerate}
\end{definition}

We recall that the minimality of $m$ implies for the case $\dim(\HH)<\infty$ that the matrix $\sum_{j=0}^{m_{HC}-1} (B_c^*)^jB_H \,B_c^j$ has a non-trivial kernel (see (2.2) or (2.18) in \cite{AAC2}). But for $\dim(\HH)=\infty$, the condition \eqref{def:mHC} only implies the following weaker condition (see (14) in \cite{AchAMN25}):
\begin{equation}\label{eps-kernel}
    \forall\varepsilon>0\: \exists \, x^0=x^0(\varepsilon)\in\SS \: \forall n=0,...,m_{HC}-1\,:\, \|\sqrt{B_H} B_c^n x^0\| \le\varepsilon.
\end{equation}

We remark that, since $B_c\in\B(\HH)$, $B_c$ is hypocoercive if and only if $\Lambda(-B_c) \subset \HHH$, see Theorem \ref{ct_exponentially_stable} below.
Moreover, for an operator $B \in \B(\HH)$ we define its \emph{spectral abscissa} $\mu=\mu(B)$ as the maximal real part of a spectral value of $B$.

For the discrete-time system \eqref{linodedd}, we use the following definition, see \cite{AchAM23ELA} for the finite dimensional case.
\begin{definition}
Let  $B_d \in \B(\HH)$.
\begin{enumerate}
\item $B_d$ is called \emph{semi-contractive} if $\Vert B_d \Vert \leq 1$.
\item $B_d$ is called \emph{hypocontractive} if there exist constants $C\ge1$ and $0<\lambda<1$ such that the solutions to the iteration \eqref{linodedd} satisfy
\begin{equation}\label{discrete-decay}
  \forall x^0\in\HH,\quad \forall k\in\N_0:\quad
  \|x_k\|\le C \lambda^k \|x^0\|.
\end{equation}
\item Let $B_d$ be semi-contractive. Its \emph{hypocontractivity index} (dHC-index) $m_{dHC}$ is defined as the smallest integer $m\in\N_0$ (if it exists) such that
\begin{equation}\label{def:mHCd}
\sum_{j=0}^m (B_d^*)^j (I -B_d^* B_d) B_d^j
 \geq \kappa I \,
\end{equation}
for some $\kappa>0$.
\end{enumerate}
\end{definition}
We remark that, due to $B_d\in\B(\HH)$, $B_d$ is hypocontractive if and only if $\Lambda(B_d) \subset \D$, see Theorem \ref{dt_exponentially_stable} below.

The following characterization of the short-time behavior of the propagator norm of \eqref{linodecc} in terms of the HC-index was proven in \cite[Th.\ 2.7]{AAC2} for matrices $B_c$ and in \cite[Th.\ 4.1]{AchAMN25} for operators:
\begin{theorem}\label{short-time-char}
Let $-B_c\in \mathcal B(\HH)$ be semi-dissipative. Then $B_c$ is hypocoercive (with hypocoercivity index $m_{HC}\in\N_0$) if and only if
\begin{equation}\label{cont-decay}
  \|e^{-B_ct}\| =1-ct^a + o(t^{a})\quad \text{for }t\to0+
\end{equation}
for some $a,\,c>0$. In this case, necessarily $a=2m_{HC}+1$.

For matrices $B_c$, the remainder term in \eqref{cont-decay}can be improved to $\mathcal O(t^{a+1})$.
\end{theorem}
In \S \ref{sec:short-time} we present a much simplified proof of this result. The sharp constant $c$ has already been known from \cite[Th.\ 2.7]{AAC2} for matrices and \cite[Th.\ 4.1]{AchAMN25} for operators. Its formula is included below in Theorem \ref{thm_propagator_norm_asymptotics} (for operators) and in Corollary \ref{thm_propagator_norm_asymptotics_finite_dim} (for matrices).

\section{Continuous-time systems: new proof of short-time decay}\label{sec:short-time}

In this section we present a new, significantly shortened proof of Theorem \ref{short-time-char}. While the upper bound for the propagator norm $\|e^{-B_c t}\|$ was based on case distinctions (see Lemma A.1 in \cite{AAC2} and \S4.2 in \cite{AchAMN25}), this is not needed here anymore. Moreover, for bounded operators on a Hilbert space, the proof was detailed in \cite{AchAMN25} only for the case of hypocoercivity index $m_{HC}=1$, while all cases are included here. Also, the new proof extends nicely to the time-discrete case analyzed in \S\ref{sec:t-discrete}.

We start with an auxiliary lemma concerning a minimization problem which plays a central role in establishing the main results of this section.
\begin{lemma}\label{min_of_quadratic_form}
Let $m \in \N$. Then
\begin{align}\label{eq:central_bilinear_form}
\min_{\lambda_1, \dots, \lambda_m \in \C} &\left[\sum_{n=0}^{m-1} \sum_{k=0}^{m-1} \frac{(-1)^{k+n}}{(k+n+1)!} \binom{k+n}{k} \lambda_{m-n} \overline{\lambda_{m-k}} + 2 \sum_{n=0}^{m-1} \frac{(-1)^{m+n}}{(m+n+1)!} \binom{m+n}{m} \Re ( \lambda_{m-n} ) \right. \\
& \left. + \frac{1}{(2m+1)!} \binom{2m}{m} \right]= \frac{1}{(2m+1)!\binom{2m}{m}}.
\end{align}
\end{lemma}

\begin{proof}
We first note that the expression we want to minimize is real-valued, so it makes sense to consider the infimum. Moreover, the minimization problem can be rewritten as a constrained minimization of a quadratic form given by
\begin{align}
&\inf_{\lambda_1,\dots,\lambda_m\in\C}
\Biggl[
\sum_{n=0}^{m-1}\sum_{k=0}^{m-1}
\frac{(-1)^{k+n}}{(k+n+1)!}\binom{k+n}{k}
\lambda_{m-n}\overline{\lambda_{m-k}}
\nonumber\\
&\qquad\qquad
+2\sum_{n=0}^{m-1}
\frac{(-1)^{m+n}}{(m+n+1)!}\binom{m+n}{m}
\Re(\lambda_{m-n})
\Biggr]
\nonumber\\
&\qquad
+\frac{1}{(2m+1)!}\binom{2m}{m}
\nonumber\\
&=
\inf_{\substack{\lambda_0,\dots,\lambda_m\in\C\\
                 \lambda_0=1}}
\Biggl[
\sum_{n=0}^{m}\sum_{k=0}^{m}
\frac{(-1)^{k+n}}{(k+n+1)!}\binom{k+n}{k}
\lambda_{m-n}\overline{\lambda_{m-k}}
\Biggr].
\label{eq:constrained_minimization}
\end{align}
Using the rescaling $\lambda_p \mapsto (-1)^{m-p}(m-p)! \,\lambda_p$ and then relabeling $\lambda_p \mapsto \lambda_{m-p+1}$ yields
\begin{align}
\eqref{eq:constrained_minimization} &= \inf_{\substack{\lambda_0, \dots, \lambda_m \in \C \\ \lambda_0 = (-1)^m/ m!}} \left[\sum_{n=0}^{m} \sum_{k=0}^{m} \frac{1}{k + n + 1} \lambda_{m-n}\overline{\lambda_{m-k}}\right] = \inf_{\substack{\lambda_1, \dots, \lambda_{m} \in \C \\ \lambda_{m+1} = (-1)^m/ m!}} \left[\sum_{n=0}^{m} \sum_{k=0}^{m} \frac{1}{k + n + 1} \lambda_{n+1}\overline{\lambda_{k+1}}\right] \nonumber \\
&= \inf_{\substack{\lambda_1, \dots, \lambda_{m} \in \C \\ \lambda_{m+1} = (-1)^m/m!}} \overline{\lambda}^T H_{m+1} \lambda, \label{eq:constrained_quadratic_form_minimization}
\end{align}
where $H_{m+1}$ is the (m+1)-dimensional Hilbert matrix and $\lambda=[\lambda_1,...,\lambda_{m+1}]^T$. As $H_{m+1}$ is known to be positive definite \cite[III]{hilbert_matrix}, the minimization can be restricted to real valued vectors $\lambda$, since
\begin{align*}
    \overline{\lambda}^T H_{m+1} \lambda &= \left( \Re(\lambda) - i \Im(\lambda) \right) ^T H_{m+1} \left( \Re(\lambda) + i \Im(\lambda) \right) = \Re(\lambda)^T H_{m+1} \Re(\lambda) + \Im(\lambda)^T H_{m+1} \Im(\lambda) \\
    &\geq \Re(\lambda)^T H_{m+1} \Re(\lambda),
\end{align*}
where the real part of the vector is taken componentwise. Moreover, the positive definiteness of $H_{m+1}$ implies that a minimizer $\lambda_{min}\in\R^{m+1}$ exists. We proceed using constrained minimization via Lagrange multipliers. From this it follows that the gradient at $\lambda_{min}$ is a multiple of the gradient of the constraint, i.e., $H_{m+1}\lambda_{min}$ is a multiple of $e_{m+1}$ (the $m+1$-st unit vector), which means that $\lambda_{min}$ is a multiple of $H_{m+1}^{-1} e_{m+1}$. Together with the given constraint this readily yields
\begin{align*}
\lambda_{min} = \frac{ H_{m+1}^{-1} e_{m+1} }{ (-1)^m m! \left( H_{m+1}^{-1} \right)_{m+1,m+1} }.
\end{align*}
Inserting this into \eqref{eq:constrained_quadratic_form_minimization} yields
\begin{align*}
\inf_{\substack{\lambda_1, \dots, \lambda_m \in \C \\ \lambda_{m+1} = (-1)^m / m!}} \overline{\lambda}^T H_{m+1} \lambda &= \overline{\lambda_{min}}^T H_{m+1} \lambda_{min} = \frac{ 1 }{ (m!)^2 \left( H_{m+1}^{-1} \right)_{m+1,m+1} } = \frac{ 1 }{ (m!)^2 (2m+1) \binom{2m}{m}^2 } \\
&= \frac{ 1 }{ (2m+1)! \binom{2m}{m} },
\end{align*}
where we used the explicit formula for $H_{m+1}^{-1}$ \cite[I.]{hilbert_matrix}.
\end{proof}

In the following, we assume that $B_c\in \mathcal B(\HH)$ with Hermitian part  $B_H \geq 0$.

\begin{lemma}\label{lem_propagator_simplification}
Let $m \in \N_0$ and let $y \colon \,\R^+\to\HH$ satisfy $\|y(t)\|=\Theta(1)$ and $\Vert \sqrt{B_H} B_c^n y(t) \Vert = o(t^{m-1-n})$ as $t\to0+$ for $0 \leq n < m$. Then
\begin{align*}
\left\Vert e^{-B_c t} \frac{y(t)}{\Vert y(t) \Vert} \right\Vert^2 = 1 - \frac{2}{\Vert y(t) \Vert^2}t \sum_{n=0}^{m} \sum_{k=0}^{m} \frac{(-t)^{k+n}}{(k+n+1)!} \binom{k+n}{k} \left\langle \sqrt{B_H} B_c^n y(t), \sqrt{B_H} B_c^k y(t) \right\rangle + o(t^{2m+1})
\end{align*}
for $t\to0+$.
\end{lemma}
\begin{proof}
Using \cite[(A.3), (A.6)]{AAC2} (i.e., the Taylor expansion  of $e^{-B_c t}$ and the Cauchy product), we calculate
\begin{align*}
\left\Vert e^{-B_c t} \frac{y(t)}{\Vert y(t) \Vert} \right\Vert^2
&= \left\langle e^{-B_c^*t} e^{-B_ct} \frac{y(t)}{\|y(t)\|},\,\frac{y(t)}{\|y(t)\|} \right\rangle \\
&= 1 +  \frac{2}{\|y(t)\|^2}
\sum_{j=1}^{2m+1} \frac{(-t)^j}{j!} \sum_{k=0}^{j-1}  \binom{j-1}{k} \left\langle (B_c^*)^k {B_H} B_c^{j-1-k} y(t), y(t) \right\rangle + \mathcal O(t^{2m+2}) \\
&= 1 + \frac{2}{\|y(t)\|^2} \sum_{k=0}^{2m} \sum_{j=k+1}^{2m+1} \frac{(-t)^j}{j!} \binom{j-1}{k} \left\langle \sqrt{B_H} B_c^{j-1-k} y(t), \sqrt{B_H} B_c^k y(t) \right\rangle + \mathcal O(t^{2m+2}) \\
&= 1 + \frac{2}{\|y(t)\|^2} \sum_{k=0}^{2m} \sum_{n=0}^{2m-k} \frac{(-t)^{n+1+k}}{(n+1+k)!} \binom{n+k}{k} \left\langle \sqrt{B_H} B_c^n y(t), \sqrt{B_H} B_c^k y(t) \right\rangle + \mathcal O(t^{2m+2})\\
&= 1 + \frac{2}{\|y(t)\|^2} \sum_{k=0}^{m} \sum_{n=0}^{2m-k} \frac{(-t)^{n+1+k}}{(n+1+k)!} \binom{n+k}{k} \left\langle \sqrt{B_H} B_c^n y(t), \sqrt{B_H} B_c^k y(t) \right\rangle + o(t^{2m+1}) ,
\end{align*}
where we have used that $\Vert \sqrt{B_H} B_c^{{k}} y(t) \Vert = \mathcal O(1)$ as $t\to0+$, for ${k >} m$,  which follows from the assumption $\|y(t)\|=\mathcal O(1)$. Analogously, we reduce the sum further as
\begin{align*}
\left\Vert e^{-B_ct} \frac{y(t)}{\Vert y(t) \Vert} \right\Vert^2
&= 1 + \frac{2}{\|y(t)\|^2} \sum_{k=0}^{m} \sum_{n=0}^{m} \frac{(-t)^{n+1+k}}{(n+1+k)!} \binom{n+k}{k} \left\langle \sqrt{B_H} B_c^n y(t), \sqrt{B_H} B_c^k y(t) \right\rangle + o(t^{2m+1}) \\
&= 1 - \frac{2}{\|y(t)\|^2}t \sum_{k=0}^{m} \sum_{n=0}^{m} \frac{(-t)^{n+k}}{(n+1+k)!} \binom{n+k}{k} \left\langle \sqrt{B_H} B_c^n y(t), \sqrt{B_H} B_c^k y(t) \right\rangle + o(t^{2m+1}),
\end{align*}
which is the claimed result.
\end{proof}

If $m$ is the hypocoercivity index of $B_c$, the conditions on $y(t)$ in Lemma \ref{lem_propagator_simplification} can be satisfied as follows.
For $m\ge1$, choose $y(t)=x^0(t^m)$ from \eqref{eps-kernel} for $t>0$. Then $\|y(t)\|=1$ and, for every $0\le n<m$, $\|\sqrt{B_H}B_c^n y(t)\|\le t^m=o(t^{m-1-n})$. For $m=0$, %take
choose $y(t)$ constant to any fixed unit vector.

Note further that in Lemma~\ref{lem_propagator_simplification} the function $y$ does not need to be continuous in $t$.

The following theorem yields the (as it turns out) optimal lower bound for the asymptotic expansion of the propagator norm, assuming essentially that the HC-index of $B_c$ is at least $m$.

\begin{theorem}\label{operator_norm_lower_bound}
Let $-B_c\in \mathcal B(\HH)$ be semi-dissipative, and let $B_H$ be the Hermitian part of $B_c$.
Suppose that, for some $m \in \N_0$,  the coercivity condition
\begin{align}\label{eq:coercivity_assumption}
    \sum_{j=0}^{m-1} (B_c^*)^j B_H B_c^j \geq \kappa I
\end{align}
is \emph{violated} for all $\kappa > 0$. Then
\begin{align}\label{eq:lininf}
\Vert e^{-B_c t} \Vert^2 \geq 1 - t^{2m+1} \frac{2}{(2m+1)!\binom{2m}{m}} \lim_{\delta \to 0+}\ \inf_{\substack{\Vert \sqrt{B_H} B_c^p y \Vert \leq \delta,\ 0 \leq p < m \\ \Vert y \Vert = 1}} \left\Vert \sqrt{B_H} B_c^m y \right\Vert^2 + o(t^{2m+1})
\end{align}
for $t\to0+$.
\end{theorem}
Note that for $m=0$ the assumption \eqref{eq:coercivity_assumption} is void and the infimum in \eqref{eq:lininf} is just taken over $\SS$. For $m>0$ these infima are never taken over the empty set, due to \eqref{eps-kernel}. Moreover, their limit as $\delta\to0+$ is finite, since $\sqrt{B_H}B_c^m\in\B(\HH)$. Further note that in view of \eqref{def:mHC}, the assumption \eqref{eq:coercivity_assumption} implies that $m_{HC} \geq m$, if the HC-index of $B_c$ exists.
\begin{proof}
We consider the sets
\begin{align*}
S_t \coloneqq
\{ y \in \SS : \Vert \sqrt{B_H} B_c^p y \Vert \leq t^{m+1-p},\ 0 \leq p < m \},
\end{align*}
which are nonempty for every $t > 0$ due to assumption \eqref{eq:coercivity_assumption}. Consequently, for each $t>0$ we can choose an element $y_0(t) \in S_t$ such that
\begin{align}\label{Bmy0-conv}
\left\lvert \left\Vert \sqrt{B_H} B_c^m y_0(t) \right\Vert - \inf_{y \in S_t} \left\Vert \sqrt{B_H} B_c^m y \right\Vert  \right\rvert < t.
\end{align}
Then we have
\begin{align}\label{eq:decay_condition}
\left\Vert \sqrt{B_H} B_c^p y_0(t) \right\Vert = \mathcal{O}(t^{m+1-p})
\end{align}
for $0 \leq p < m$, and more generally we have the (non-sharp) estimate
\begin{align}\label{eq:decay_condition_2}
\left\Vert \sqrt{B_H} B_c^p y_0(t) \right\Vert = \mathcal{O}(t^{m-p})
\end{align}
for any $p \in \N_0$ (because $\Vert \sqrt{B_H} B_c^p y_0(t) \Vert = \mathcal{O}(1)$ for any $p \geq m$). From \eqref{Bmy0-conv} we have
\begin{align}\label{eq:infimum_convergence}
\lim_{t \to 0+} \left\Vert \sqrt{B_H} B_c^m y_0(t) \right\Vert = \lim_{t \to 0+}\ \inf_{y \in S_t} \left\Vert \sqrt{B_H} B_c^m y \right\Vert = \lim_{\delta \to 0+}\ \inf_{\substack{\Vert \sqrt{B_H} B_c^p y \Vert \leq \delta,\, 0 \leq p < m \\ \Vert y \Vert = 1}} \left\Vert \sqrt{B_H} B_c^m y \right\Vert,
\end{align}
where the limits in the middle and on the right-hand side are well-defined as the infima are nondecreasing with respect to decreasing $t$ and $\delta$, respectively, and bounded from above by $\|\sqrt{B_H}B_c^m\|$.
The equality of the last two limits follows from the enclosure
\[
  \inf_{\substack{\Vert \sqrt{B_H} B_c^p y \Vert \leq t^2,\, 0 \leq p < m \\ \Vert y \Vert = 1}} \left\Vert \sqrt{B_H} B_c^m y \right\Vert
  \le\inf_{y \in S_t} \left\Vert \sqrt{B_H} B_c^m y \right\Vert
  \le \inf_{\substack{\Vert \sqrt{B_H} B_c^p y \Vert \leq t^{m+1},\, 0 \leq p < m \\ \Vert y \Vert = 1}} \left\Vert \sqrt{B_H} B_c^m y \right\Vert
\]
for $0<t\le1$, and hence choosing $\delta=t^2$ and, respectively, $\delta=t^{m+1}$ in \eqref{eq:infimum_convergence}.

For $m=0$, take $y(t):=y_0(t)$. Since $\|y_0(t)\|=1$,
Lemma \ref{lem_propagator_simplification} gives
\[
\|e^{-B_ct}y_0(t)\|^2
=1-2t\|\sqrt{B_H}y_0(t)\|^2+o(t).
\]
Together with \eqref{eq:infimum_convergence} and
$\|e^{-B_ct}\|^2\ge\|e^{-B_ct}y_0(t)\|^2$, this proves
\eqref{eq:lininf} for $m=0$.

For the remainder of the proof, assume $m\ge1$. Choose $\lambda_1, \dots, \lambda_m \in \R$ as the unique minimizer of the quadratic form \eqref{eq:central_bilinear_form} and consider
\begin{equation}\label{y(t)}
y(t) \coloneqq y_0(t) + \sum_{p=1}^m \lambda_p\, (tB_c)^p y_0(t),\,t>0.
\end{equation}
Then,
\begin{align}\label{eq:decay_condition_3}
\left\Vert \sqrt{B_H} B_c^n (tB_c)^p y_0(t) \right\Vert = \mathcal{O}(t^{m+1-n})
\end{align}
for $0 \leq p,n$ with $p+n < m$ due to \eqref{eq:decay_condition} and
\begin{align}\label{eq:decay_condition_4}
\left\Vert \sqrt{B_H} B_c^n (tB_c)^p y_0(t) \right\Vert = \mathcal{O}(t^{m-n})
\end{align}
for any $0 \leq p,n$ due to \eqref{eq:decay_condition_2}. Consequently, $\Vert \sqrt{B_H} B_c^n y(t) \Vert = o(t^{m-1-n})$ and $\|y(t)\|=\Theta(1)$. Thus, Lemma \ref{lem_propagator_simplification} yields
\begin{align}\label{eq:lower_bound_expression}
\left\Vert e^{-B_ct} \frac{y(t)}{\Vert y(t) \Vert} \right\Vert^2 = 1 - \frac{2}{\|y(t)\|^2} t \sum_{n=0}^{m} \sum_{k=0}^{m} \frac{(-t)^{k+n}}{(k+n+1)!} \binom{k+n}{k} \left\langle \sqrt{B_H} B_c^n y(t), \sqrt{B_H} B_c^k y(t) \right\rangle + o(t^{2m+1})
\end{align}
as $t\to0+$.

When expanding $y(t)$ in the inner product of \eqref{eq:lower_bound_expression},
i.e.,\ in
\begin{equation}\label{inner-prod}
  \sum_{p=0}^m\sum_{q=0}^m \lambda_p\lambda_q \left\langle \sqrt{B_H}B_c^n(tB_c)^p y_0(t), \sqrt{B_H}B_c^k(tB_c)^q y_0(t) \right\rangle,
\end{equation}
for the sum in \eqref{y(t)}, where we set $\lambda_0 \coleq 1$, we need to distinguish several cases when estimating the order of \eqref{inner-prod}:

For $n+p < m$ we have the estimate
\begin{align*}
\Vert \sqrt{B_H} B_c^n (tB_c)^p y_0(t) \Vert = \mathcal{O}(t^{m+1-n})
\end{align*}
using \eqref{eq:decay_condition_3}, and
\begin{align*}
\Vert \sqrt{B_H} B_c^k (tB_c)^q y_0(t) \Vert = \mathcal{O}(t^{m-k})
\end{align*}
using \eqref{eq:decay_condition_4} for any $k,\,q$. Together these estimates  yield
\begin{align*}
t^{1+k+n} \left\lvert \left\langle \sqrt{B_H} B_c^n (tB_c)^p y_0(t), \sqrt{B_H} B_c^k (tB_c)^q y_0(t) \right\rangle \right\rvert = \mathcal{O}(t^{2m+2}),
\end{align*}
contributing only to the remainder term in \eqref{eq:lower_bound_expression}.
This last estimate also holds if $k+q < m$ and $n,p$ are arbitrary. In the case $n+p \geq m+1$ and $k+q \geq m$ (or $n+p \geq m$ and $k+q \geq m+1$), we get
\begin{align}\label{remain2}
&t^{1+k+n} \left\lvert \left\langle \sqrt{B_H} B_c^n (tB_c)^p y_0(t), \sqrt{B_H} B_c^k (tB_c)^q y_0(t) \right\rangle \right\rvert \\
&= t^{1+k+n+p+q} \left\lvert \left\langle \sqrt{B_H} B_c^{n+p} y_0(t), \sqrt{B_H} B_c^{k+q} y_0(t) \right\rangle \right\rvert = \mathcal{O}(t^{2m+2}) \nonumber
\end{align}
using $\Vert \sqrt{B_H} B_c^{n+p} y_0(t) \Vert = \mathcal{O}(1)$ and $\Vert \sqrt{B_H} B_c^{k+q} y_0(t) \Vert = \mathcal{O}(1)$.
So, \eqref{remain2} also contributes only to the remainder term in \eqref{eq:lower_bound_expression}.

To sum up, if we expand $y(t)$ in \eqref{eq:lower_bound_expression},
only the terms with $n+p=k+q=m$ may \emph{not} be part of the remainder term. Hence, \eqref{eq:lower_bound_expression} simplifies to
\begin{align*}
&\left\Vert e^{-B_ct} \frac{y(t)}{\Vert y(t) \Vert} \right\Vert^2 = 1 - \frac{2}{\|y(t)\|^2}t^{2m+1} \bigg( \sum_{n=0}^{m-1} \sum_{k=0}^{m-1} \frac{(-1)^{k+n}}{(k+n+1)!} \binom{k+n}{k} \lambda_{m-n} {\lambda_{m-k}} \\
& \qquad + 2 \sum_{n=0}^{m-1} \frac{(-1)^{m+n}}{(m+n+1)!} \binom{m+n}{m}   \lambda_{m-n}  + \frac{1}{(2m+1)!} \binom{2m}{m} \bigg)\left\Vert \sqrt{B_H} B_c^m y_0(t) \right\Vert^2 + o(t^{2m+1}) \\
&\qquad = 1 -2t^{2m+1} \bigg( \sum_{n=0}^{m-1} \sum_{k=0}^{m-1} \frac{(-1)^{k+n}}{(k+n+1)!} \binom{k+n}{k} \lambda_{m-n} {\lambda_{m-k}} \\
& \qquad + 2 \sum_{n=0}^{m-1} \frac{(-1)^{m+n}}{(m+n+1)!} \binom{m+n}{m}  \lambda_{m-n} + \frac{1}{(2m+1)!} \binom{2m}{m} \bigg)\left\Vert \sqrt{B_H} B_c^m y_0(t) \right\Vert^2 + o(t^{2m+1}),
\end{align*}
where we have used $\frac{1}{\|y(t)\|^2} = 1+\mathcal O(t)$ in the last equality. By construction of $\lambda_1, \dots, \lambda_m$ (see Lemma \ref{min_of_quadratic_form}) this yields
\begin{align*}
\left\Vert e^{-B_c t} \frac{y(t)}{\Vert y(t) \Vert} \right\Vert^2 = 1 - t^{2m+1} \frac{2}{(2m+1)!\binom{2m}{m}} \left\Vert \sqrt{B_H} B_c^m y_0(t) \right\Vert^2 + o(t^{2m+1}),
\end{align*}
and due to \eqref{eq:infimum_convergence} (using $\| \sqrt{B_H} B_c^m y_0(t) \|^2=\lim_{t\to0+} \| \sqrt{B_H} B_c^m y_0(t) \|^2+o(1)$) we get
\begin{align*}
\left\Vert e^{-B_c t} \frac{y(t)}{\Vert y(t) \Vert} \right\Vert^2 = 1 - t^{2m+1} \frac{2}{(2m+1)!\binom{2m}{m}}
\lim_{\delta \to 0+}\ \inf_{\substack{\Vert \sqrt{B_H} B_c^p y \Vert \leq \delta,\, 0 \leq p < m \\ \Vert y \Vert = 1}} \left\Vert \sqrt{B_H} B_c^m y \right\Vert^2 + o(t^{2m+1}).
\end{align*}
The claim now follows immediately, since
\begin{align*}
\left\Vert e^{-B_c t} \right\Vert^2 &\geq \left\Vert e^{-B_ct} \frac{y(t)}{\Vert y(t) \Vert} \right\Vert^2. \qedhere
\end{align*}
\end{proof}

\begin{remark}\label{asymptotics_for_propagator_norm}
The lower bound in Theorem \ref{operator_norm_lower_bound} already implies that, if $B_c$ has hypocoercivity index greater than or equal to $m \in \N_0$, then $\Vert e^{-B_c t} \Vert = 1 - \mathcal{O}(t^{2m+1})$.
\end{remark}

\begin{lemma}\label{bilinear_form_minimization}
Let $z_0, \dots, z_m \in \HH$ with some $m\in\N_0$ fixed. Then there exists $C_m > 0$ (independent of $z_0,...,z_m$) such that
\begin{align*}
\sum_{n=0}^m \sum_{k=0}^m \frac{(-1)^{k+n}}{(k+n+1)!} \binom{k+n}{k} \left\langle z_{m-n}, z_{m-k} \right\rangle \geq C_m \left\Vert z_j \right\Vert^2
\end{align*}
for all $0 \leq j \leq m$ and for $j=0$ we explicitly obtain
\begin{align}\label{0-bound}
\sum_{n=0}^m \sum_{k=0}^m \frac{(-1)^{k+n}}{(k+n+1)!} \binom{k+n}{k} \left\langle z_{m-n}, z_{m-k} \right\rangle \geq \frac{1}{(2m+1)!\binom{2m}{m}} \left\Vert z_0 \right\Vert^2.
\end{align}
\end{lemma}
\begin{proof}
Let $0 \leq j \leq m$ be arbitrary but fixed, and let
\begin{align}\label{eq:orthogonal_decomposition}
z_p = \lambda_p z_j + \tilde{z}_p
\end{align}
be an orthogonal decomposition, i.e., $\langle \tilde{z}_p, z_j \rangle = 0$ for every $0 \leq p \leq m$. Then
\begin{align}\label{eq:bilinear_form}
&\sum_{n=0}^m \sum_{k=0}^m \frac{(-1)^{k+n}}{(k+n+1)!} \binom{k+n}{k} \left\langle z_{m-n}, z_{m-k} \right\rangle \nonumber \\
&= \sum_{n=0}^m \sum_{k=0}^m \frac{(-1)^{k+n}}{(k+n+1)!} \binom{k+n}{k} \left\langle \lambda_{m-n} z_j + \tilde{z}_{m-n}, \lambda_{m-k} z_j + \tilde{z}_{m-k} \right\rangle \nonumber \\
&= \sum_{n=0}^m \sum_{k=0}^m \frac{(-1)^{k+n}}{(k+n+1)!} \binom{k+n}{k} \lambda_{m-n} \overline{\lambda_{m-k}} \Vert z_j \Vert^2 + \sum_{n=0}^m \sum_{k=0}^m \frac{(-1)^{k+n}}{(k+n+1)!} \binom{k+n}{k} \left\langle \tilde{z}_{m-n}, \tilde{z}_{m-k} \right\rangle.
\end{align}
Now we note that
\begin{align*}
\left[ \frac{(-1)^{k+n}}{(k+n+1)!} \binom{k+n}{k}\right]_{n,k=0,...,m}\ge C_m \Id,
\end{align*}
since it is just a congruently transformed Hilbert matrix (compare with \eqref{eq:constrained_minimization}). Consequently,
\begin{align*}
\sum_{n=0}^{m} \sum_{k=0}^{m} \frac{(-1)^{k+n}}{(k+n+1)!} \binom{k+n}{k} \langle \tilde{z}_{m-n}, \tilde{z}_{m-k} \rangle \geq 0,
\end{align*}
being the trace of the product of two Hermitian, positive semidefinite matrices. Moreover,
\begin{align*}
\sum_{n=0}^m \sum_{k=0}^m \frac{(-1)^{k+n}}{(k+n+1)!} \binom{k+n}{k} \lambda_{m-n} \overline{\lambda_{m-k}} \Vert z_j \Vert^2 \geq C_m \Vert z_j \Vert^2,
\end{align*}
where we additionally have used $\lambda_j = 1$ (due to \eqref{eq:orthogonal_decomposition}). Using these lower bounds in \eqref{eq:bilinear_form} yields
\begin{align*}
\sum_{n=0}^m \sum_{k=0}^m \frac{(-1)^{k+n}}{(k+n+1)!} \binom{k+n}{k} \left\langle z_{m-n}, z_{m-k} \right\rangle \geq C_m \Vert z_j \Vert^2.
\end{align*}
The last claim, i.e.\ \eqref{0-bound}, follows from using $\lambda_0=1$ in the first sum of \eqref{eq:bilinear_form} and applying Lemma \ref{min_of_quadratic_form} in the last step:
\begin{align*}
&\sum_{n=0}^m \sum_{k=0}^m \frac{(-1)^{k+n}}{(k+n+1)!} \binom{k+n}{k} \lambda_{m-n} \overline{\lambda_{m-k}} \Vert z_0 \Vert^2 \\
&= \bigg[ \sum_{n=0}^{m-1} \sum_{k=0}^{m-1} \frac{(-1)^{k+n}}{(k+n+1)!} \binom{k+n}{k} \lambda_{m-n} \overline{\lambda_{m-k}} + 2 \sum_{n=0}^{m-1} \frac{(-1)^{m+n}}{(m+n+1)!} \binom{m+n}{m} \Re ( \lambda_{m-n} ) \\
& \qquad + \frac{1}{(2m+1)!} \binom{2m}{m} \bigg] \Vert z_0 \Vert^2 \\
&\geq \frac{1}{(2m+1)!\binom{2m}{m}} \left\Vert z_0 \right\Vert^2. \qedhere
\end{align*}
\end{proof}

\begin{lemma}\label{decay_properties}
Let $B_c\in \mathcal B(\HH)$ have hypocoercivity index greater than or equal to some $m \in \N_0$ and let $y \colon \R^+ \to \HH$ satisfy $\Vert y(t) \Vert = 1$ for $t \in \R^+$ and $\Vert e^{-B_ct} y(t) \Vert - \Vert e^{-B_c t} \Vert = o(t^{2m+1})$ for $t \to 0+$. Then $\Vert \sqrt{B_H} B_c^{p} y(t) \Vert = o(t^{m-1-p})$ for $0 \leq p < m$.
\end{lemma}
%Note that the above claim does \emph{not} necessarily hold for $t=0$, since $y$ was not assumed to be continuous there.

\begin{proof}
We proceed by induction over $m$. Concerning $m=0$ one could choose $y(t)\equiv y_0$ for any $y_0\in\SS$ but, anyhow, there is nothing to show.

Next, we assume that the statement holds for $m$ and we want to show the corresponding statement for $m+1$. So let us assume that $B_c$ has hypocoercivity index greater than or equal to $m+1$ and $y \colon \R^+ \to \HH$ satisfies $\Vert y(t) \Vert = 1$ for $t \in \R^+$ and $\Vert e^{-B_c t} y(t) \Vert - \Vert e^{-B_c t} \Vert = o(t^{2m+3})$ for $t \to 0+$. Then we can use the induction hypothesis to deduce that $\Vert \sqrt{B_H} B_c^{p} y(t) \Vert = o(t^{m-1-p})$ for $0 \leq p < m$ and thus we can expand the squared norm using Lemma \ref{lem_propagator_simplification}:
\begin{align*}
\Vert e^{-B_c t} y(t) \Vert^2 &= 1 - 2t \sum_{n=0}^{m} \sum_{k=0}^{m} \frac{(-t)^{k+n}}{(k+n+1)!} \binom{k+n}{k} \left\langle \sqrt{B_H} B_c^n y(t), \sqrt{B_H} B_c^k y(t) \right\rangle + o(t^{2m+1}) \\
&= 1 - 2t\sum_{n=0}^{m} \sum_{k=0}^{m} \frac{(-1)^{k+n}}{(k+n+1)!} \binom{k+n}{k} \left\langle t^n \sqrt{B_H} B_c^n y(t), t^k\sqrt{B_H} B_c^k y(t) \right\rangle + o(t^{2m+1}) \\
&\leq 1 -2t C_m t^{2p} \Vert \sqrt{B_H} B_c^p y(t) \Vert^2 + o(t^{2m+1})
\end{align*}
for any $0 \leq p < m+1$, where the last inequality with corresponding constant $C_m > 0$ follows from Lemma \ref{bilinear_form_minimization} applied to $z_{m-n}:=t^n\sqrt{B_H}B_c^ny(t)$, $n = 0, \dotsc ,m$.

Consequently, we get
\begin{align*}
2C_mt^{2p+1} \Vert \sqrt{B_H} B_c^p y(t) \Vert^2 \leq 1 - \Vert e^{-B_c t} y(t) \Vert^2 + o(t^{2m+1}).
\end{align*}
Then Remark \ref{asymptotics_for_propagator_norm} yields $1 - \Vert e^{-B_ct} \Vert = \mathcal{O}(t^{2m+3})$ and, together with the assumption that $\Vert e^{-B_ct} y(t) \Vert - \Vert e^{-B_c t} \Vert = o(t^{2m+1})$, we get
\begin{align*}
2C_mt^{2p+1} \Vert \sqrt{B_H} B_c^p y(t) \Vert^2 \leq o(t^{2m+1}).
\end{align*}
Hence,
\begin{align*}
\Vert \sqrt{B_H} B_c^p y(t) \Vert = o(t^{m-p})
\end{align*}
for $t\to 0+$ and any $0 \leq p < m+1$.
\end{proof}

In Lemma \ref{decay_properties}, a function $y(t)$ satisfying the listed condition can always be found due to the definition of the operator norm $\|e^{-B_c t}\|$ as a supremum.

The following theorem yields the optimal upper bound for the asymptotic expansion of the propagator norm.

\begin{theorem}\label{operator_norm_upper_bound}
Let $B_c\in \mathcal B(\HH)$ have hypocoercivity index $m_{HC} \in \N_0$.  Then
\begin{align*}
\Vert e^{-B_ct} \Vert^2 \leq 1 - t^{2m_{HC}+1} \frac{2}{(2m_{HC}+1)!\binom{2m_{HC}}{m_{HC}}} \lim_{\delta \to 0+}\ \inf_{\substack{\Vert \sqrt{B_H} B_c^p y \Vert \leq \delta,\, 0 \leq p < m_{HC} \\ \Vert y \Vert = 1}} \left\Vert \sqrt{B_H} B_c^{m_{HC}} y \right\Vert^2 + o(t^{2m_{HC}+1})
\end{align*}
for $t\to0+$.
\end{theorem}
Recall from Theorem \ref{operator_norm_lower_bound} that for $\mHC=0$ the infimum is just taken over $\SS$.
\begin{proof}
Choose some $y\colon\R^+ \to \HH$ satisfying $\Vert y(t) \Vert = 1$ for $t>0$, and $\Vert e^{-B_c t} y(t) \Vert - \Vert e^{-B_c t} \Vert = o(t^{2m+1})$. Here and in the remainder of this proof we abbreviate the HC-index by $m$. Then Lemma \ref{decay_properties} implies that $\Vert \sqrt{B_H} B_c^{n} y(t) \Vert = o(t^{m-1-n})$ for $0 \leq n < m$. Hence, we can apply Lemma \ref{lem_propagator_simplification} and obtain:
\begin{align*}
\Vert e^{-B_c t} y(t) \Vert^2 &= 1 - 2t \sum_{n=0}^{m} \sum_{k=0}^{m} \frac{(-t)^{k+n}}{(k+n+1)!} \binom{k+n}{k} \left\langle \sqrt{B_H} B_c^n y(t), \sqrt{B_H} B_c^k y(t) \right\rangle + o(t^{2m+1}) \\
&= 1 -2t \sum_{n=0}^m \sum_{k=0}^m \frac{(-1)^{k+n}}{(k+n+1)!} \binom{k+n}{k} \left\langle t^n \sqrt{B_H} B_c^n y(t), t^k \sqrt{B_H} B_c^k y(t) \right\rangle + o(t^{2m+1}) \\
&\leq 1-2t \frac{1}{(2m+1)!\binom{2m}{m}} t^{2m} \left\Vert \sqrt{B_H} B_c^m y(t) \right\Vert^2 + g(t) , \quad\text{with } g(t)=o(t^{2m+1}), %o(t^{2m+1}),
\end{align*}
where we have used Lemma \ref{bilinear_form_minimization} with $z_0:=t^m\sqrt{B_H}B_c^m y(t)$ for the last inequality. Consequently, for any $\delta > 0$ we can find $T_\delta > 0$ such that for any $0 < t \leq T_{\delta}$ it holds that
\begin{align*}
\Vert e^{-B_c t} y(t) \Vert^2 \leq 1 - c_\delta \,t^{2m+1}  + g(t),
\end{align*}
with
\[
  c_\delta:= \frac{2}{(2m+1)!\binom{2m}{m}} \inf_{\substack{\Vert \sqrt{B_H} B_c^p y \Vert \leq \delta,\ 0 \leq p < m \\ \Vert y \Vert = 1}} \left\Vert \sqrt{B_H} B_c^m y \right\Vert^2 .
\]
Note that $y(t)\big|_{(0,T_\delta]}$ is
included in the set in which we minimize. Moreover, the function $g$ is independent of $\delta$, $c_\delta$ increases as $\delta\to0$, but $T_\delta$ (typically) decreases to 0 as $\delta\to0+$.

Using the fact that $\Vert e^{-B_c t} y(t) \Vert - \Vert e^{-B_ct} \Vert = o(t^{2m+1})$ we obtain for every fixed $\delta>0$:
\[
  \limsup_{t\to0+} \frac{\|e^{-B_c t}\|^2-1}{t^{2m+1}} =
  \limsup_{t\to0+} \frac{\|e^{-B_c t} y(t)\|^2-1}{t^{2m+1}} \le -c_\delta.
\]
Hence,
\[
  \limsup_{t\to0+} \frac{\|e^{-B_c t}\|^2-1}{t^{2m+1}}
\le -\lim_{\delta\to0+}c_\delta,
\]
which is a reformulation of the claim.
\end{proof}

In summary, we have the following short-time decay result.
\begin{theorem}\label{thm_propagator_norm_asymptotics}
Let $-B_c\in \mathcal B(\HH)$ be semi-dissipative. Then $B_c$ is hypocoercive (with hypocoercivity index $m_{HC}\in\N_0$) if and only if
\begin{equation}\label{short-decay}
  \|e^{-B_c t}\| =1-ct^a + o(t^{a})\quad \text{for }t\to0+
\end{equation}
for some $a,\,c>0$. In this case, necessarily $a=2m_{HC}+1$ and
\begin{align}\label{c-operator}
c = \frac{1}{(2m_{HC}+1)!\,\binom{2m_{HC}}{m_{HC}}} \lim_{\delta \to 0+}\ \inf_{\substack{\Vert \sqrt{B_H} B_c^p y \Vert \leq \delta,\ 0 \leq p < m_{HC} \\ \Vert y \Vert = 1}} \left\Vert \sqrt{B_H} B_c^{m_{HC}} y \right\Vert^2.
\end{align}
\end{theorem}
\begin{proof} \leavevmode\par\noindent
$\implies$: Since the HC-index of $B_c$ equals $m_{HC}$, condition \eqref{eq:coercivity_assumption} with $m:=m_{HC}$ is violated for all $\kappa>0$. Then this direction follows immediately from Theorem \ref{operator_norm_lower_bound} and Theorem \ref{operator_norm_upper_bound}. For later reference note that this direction of the proof only used that the HC-index of $B_c$ is $m_{HC}$ as defined by the coercivity estimate \eqref{def:mHC}, but we did not use the exponential decay as in \eqref{def:HC}.

$\impliedby$:
First, let $m \coloneqq \left\lceil \frac{a-1}{2} \right\rceil$. Then the operator $\sum_{p=0}^{m} (B_c^*)^p B_H B_c^p$ must be coercive because otherwise Theorem \ref{operator_norm_lower_bound} (applied with $m+1$ in place of $m$) would imply
\begin{align*}
\Vert e^{-B_c t} \Vert^2 \geq 1 - \tilde{c} t^{2m+3} + o(t^{2m+3})
\end{align*}
for some $\tilde{c} \geq 0$, which contradicts the assumption that \eqref{short-decay} as $2m+3 > 2m+1 \geq a$. Consequently, $B_c$ has a hypocoercivity index smaller than or equal to $m$. Then the claim, i.e., that $m_{HC}(B_c)=m$ and $a=2m_{HC}+1$, follows immediately from the converse direction.

Finally, \eqref{short-decay} also implies that $B_c$ is hypocoercive (by a general semigroup result in \cite[Proposition V.1.7]{engel_nagel}; see also \cite[Proposition 4.4]{AchAMN25}). Alternatively, the hypocoercivity of $B_c$, i.e., the exponential decay of the corresponding semigroup, can be obtained directly by considering the Lyapunov functional $\langle x,P x\rangle$ with the bounded, symmetric, coercive operator $P:=\sum_{j=0}^m (B_c^*)^j B_c^j$, see \cite[\S2]{AchAMN25b} for details.
\end{proof}

For matrices, Theorem~\ref{thm_propagator_norm_asymptotics} simplifies in two ways, the formula for the constant $c$ and the order of the remainder term:
\begin{corollary}\label{thm_propagator_norm_asymptotics_finite_dim}
Let $-B_c\in \mathcal B(\HH)$ be semi-dissipative and let $\HH$ be finite-dimensional. Then $B_c$ is hypocoercive (with hypocoercivity index $m_{HC}\in\N_0$) if and only if
\begin{align}\label{eq:finite_dim_propagator_norm}
\|e^{-B_c t}\| =1-ct^a + \mathcal{O}(t^{a+1})\quad \text{for }t\to 0+,
\end{align}
for some $a,\,c>0$. In this case, necessarily $a=2m_{HC}+1$ and
\begin{align}\label{c-matrix}
c = \frac{1}{(2m_{HC}+1)!\,\binom{2m_{HC}}{m_{HC}}} \min_{\substack{y \in \ker(\sqrt{B_H} B_c^p),\, 0 \leq p < m_{HC} \\ \Vert y \Vert = 1}} \left\Vert \sqrt{B_H} B_c^{m_{HC}} y \right\Vert^2.
\end{align}
\end{corollary}
\begin{proof}
The main part of this equivalence just carries over from Theorem~\ref{thm_propagator_norm_asymptotics}. So we only have to prove the formula for the constant and the order of the remainder term. In the remainder of this proof we abbreviate the HC-index of $B_c$ by $m$.

First, note that due to \cite[Lemma 1]{KOHAUPT20011} there exists $t_0 > 0$ such that the mapping $t \mapsto \Vert e^{-B_c t} \Vert$ is analytic on $[0, t_0)$, which directly implies the improved remainder term in \eqref{eq:finite_dim_propagator_norm}.

Second, for any $\delta > 0$ we have
\begin{align*}
\inf_{\substack{\Vert \sqrt{B_H} B_c^p y \Vert \leq \delta,\, 0 \leq p < m \\ \Vert y \Vert = 1}} \left\Vert \sqrt{B_H} B_c^m y \right\Vert^2 \leq \inf_{\substack{y \in \ker(\sqrt{B_H} B_c^p),\, 0 \leq p < m \\ \Vert y \Vert = 1}} \left\Vert \sqrt{B_H} B_c^m y \right\Vert^2,
\end{align*}
and thus
\begin{align*}
\lim_{\delta \to 0+}\inf_{\substack{\Vert \sqrt{B_H} B_c^p y \Vert \leq \delta,\, 0 \leq p < m \\ \Vert y \Vert = 1}} \left\Vert \sqrt{B_H} B_c^m y \right\Vert^2 \leq \inf_{\substack{y \in \ker(\sqrt{B_H} B_c^p),\, 0 \leq p < m \\ \Vert y \Vert = 1}} \left\Vert \sqrt{B_H} B_c^m y \right\Vert^2,
\end{align*}
relating the two constants in \eqref{c-operator} and \eqref{c-matrix}.
Concerning the other direction of the estimate, for any $n \in \N$ we choose a sequence $y_n \in \SS$ with $\Vert \sqrt{B_H} B_c^p  y_n \Vert \leq \frac{1}{n}$ for all $0 \leq p < m$ (which exists due to \eqref{eps-kernel}) and with
\begin{align}\label{eq:infimizing_sequence}
\left\Vert \sqrt{B_H} B_c^m y_n \right\Vert^2 \leq \inf_{\substack{\Vert \sqrt{B_H} B_c^p y \Vert \leq \frac{1}{n},\, 0 \leq p < m \\ \Vert y \Vert = 1}} \left\Vert \sqrt{B_H} B_c^m y \right\Vert^2 + \frac{1}{n}.
\end{align}
As $\HH$ is finite-dimensional and $\Vert y_n \Vert = 1$ for all $n \in \N$, we can find a convergent subsequence $y_{n_k}$ with limit $y_* \in \HH$. Then we clearly have $\Vert y_* \Vert = 1$ and $y_* \in \ker(\sqrt{B_H} B_c^p)$ for $0 \leq p < m$. Moreover, taking the limit $n \to \infty$ in \eqref{eq:infimizing_sequence} yields
\begin{align*}
\left\Vert \sqrt{B_H} B_c^m y_* \right\Vert^2 \leq \lim_{n \to \infty} \inf_{\substack{\Vert \sqrt{B_H} B_c^p y \Vert \leq \frac{1}{n},\, 0 \leq p < m \\ \Vert y \Vert = 1}} \left\Vert \sqrt{B_H} B_c^m y \right\Vert^2.
\end{align*}
Consequently,
\begin{align*}
\inf_{\substack{y \in \ker(\sqrt{B_H} B_c^p),\, 0 \leq p < m \\ \Vert y \Vert = 1}} \left\Vert \sqrt{B_H} B_c^m y \right\Vert^2 \leq \lim_{\delta \to 0+}\ \inf_{\substack{\Vert \sqrt{B_H} B_c^p y \Vert \leq \delta,\, 0 \leq p < m \\ \Vert y \Vert = 1}} \left\Vert \sqrt{B_H} B_c^m y \right\Vert^2.
\end{align*}
Finally, we note that the infimum can be replaced by a minimum because we are taking the infimum of a continuous function over a compact set.
\end{proof}

\begin{remark}\label{rem:m=1}
We  now illustrate for the case $m_{HC}=1$ why the leading decay term in the expansion of the propagator norm has (at least) order 3: %odd order:
For matrices, the index $m_{HC}=1$ implies $\dim(\ker(B_H))\ge1$. Any solution to \eqref{linodecc} with normalized $x^0\in \ker(B_H)$ satisfies $\frac{d}{dt} \|x(t)\| \big|_{t=0} =0$. The trajectory $x(t)$ and its norm are analytic on $\R$ (also for $t<0$~!). The Taylor expansion of the norm of this trajectory cannot be of the form $\|x(t)\|=1-ct^2+\calO(t^3)$ with some $c>0$, as this would violate the semi-dissipativity for $t<0$. Thus, this trajectory norm (and hence also the propagator norm) must be of the form $1-ct^3+\calO(t^4)$ with some $c\ge0$. Note that this argument cannot be applied directly to the propagator norm, as it is in general not real analytic in a neighborhood of $t=0$, see \cite{Graf,KOHAUPT20011}.
\end{remark}

\section{Equivalence of hypocoercivity index and hypocontractivity index via the Cayley transform}
\label{sec:Cayley}

In this section we extend a result that was obtained for finite dimensional systems in \cite{AchAM23ELA} to the case of infinite dimensional systems. Namely, the hypocoercivity index of a continuous-time system coincides with the hypocontractivity index of the discrete-time system obtained by applying the implicit midpoint rule as time-discretization method.

\subsection{Continuous-time systems}
We first recall the results for continuous-time dynamical systems in Hilbert spaces.
Let $\HH$ be a separable Hilbert space, $B_c \in \B(\HH)$ with $\ker B_c=\{0\}$, and $x^0 \in \HH$ be an initial datum.
Consider the continuous-time dynamical system~\eqref{linodecc}.
Then one has the following classical result about exponential stability, see e.g. Theorem I. 3.14, \cite{engel_nagel}.
\begin{theorem}\label{ct_exponentially_stable}
The zero solution $x\equiv 0$ of \eqref{linodecc} is exponentially stable if and only if $\Lambda(-B_c) \subset \HHH$.
\end{theorem}
In Lemma 2 of \cite{AchAM25} the following result was shown.
\begin{lemma}
\label{hypocoercivity_equivalent_conditions}
Let $-B_c=B_S-B_H\in \mathcal B(\HH)$, with $B_S=-B_S^*$ and $B_H=B_H^*$, be semi-dissipative, i.e., it satisfies $B_H\geq 0$. Then the following conditions are equivalent:
\begin{enumerate}
\item There exists $m \in \N_0$ such that
\begin{align*}
\Span \left( \bigcup_{j=0}^m \im \left( (B^*_c)^j \sqrt{B_H} \right) \right) = \HH.
\end{align*}
\item There exists $m \in \N_0$ such that
\begin{align*}
\bigcap_{j=0}^m \ker \left( \sqrt{B_H} B_c^j \right) = \{ 0 \} \ \ \text{and} \ \ \Span \left( \bigcup_{j=0}^m \im \left( (B^*_c)^j \sqrt{B_H} \right) \right) \text{ is closed}.
\end{align*}
\item There exists $m \in \N_0$ such that
\begin{align*}
\sum_{j=0}^m (B_c^*)^j \,B_H\, B_c^j \geq \kappa I
\end{align*}
for some $\kappa > 0$.
\end{enumerate}
Moreover, if any of these conditions is satisfied for some $m \in \N_0$, then all conditions are satisfied for this $m$.
\end{lemma}

After recalling the results for the continuous-time case, in the next subsection we derive the corresponding results in the discrete-time case.
\subsection{Discrete-time dynamical systems}
Consider the discrete-time dynamical system~\eqref{linodedd} with $B_d\in\B(\HH)$,  $1\not\in \Lambda(B_d)$ (in analogy to $0\not\in\Lambda(B_c)$), and an initial datum $x^0 \in \HH$.
The system is called \emph{exponentially stable} (or, equivalently, $B_d$ is called \emph{hypocontractive}) if there exist constants $C\ge1$ and $0<\lambda<1 $ such that for all $x^0$ and all $k\in \mathbb N_0$ one has
\[
\| x_k\|\leq C \lambda^k \|x^0\|.
\]
One has the following classical result about
exponential stability (see e.g Proposition II. 1.3 in \cite{eisner_operator_semigroups} after noting that $\Lambda(B_d) \subset \D$ is equivalent to $\rho(B_d) < 1$).
\begin{theorem}\label{dt_exponentially_stable}
The zero solution of $\eqref{linodedd}$ (i.e. $x_k=0$ for $k\in \N_0$)  is exponentially stable if and only if $\Lambda(B_d) \subset \D$.
\end{theorem}
One has the following result which follows directly from Lemma 2 in \cite{AchAM25} by setting there $R \coloneqq \Id - B_d^* B_d$ and $J \coloneqq B_d^*$.
The assumption that $J$ has to be skew-adjoint can be removed without any adjustments to the proof.

\begin{lemma}\label{hypocontractivity_equivalent_conditions}
Let $B_d\in \B(\HH)$ be semi-contractive. Then the following conditions are equivalent:
\begin{enumerate}
\item There exists $m \in \N_0$ such that
\begin{align*}
\Span \left( \bigcup_{j=0}^m \im \left( (B_d^*)^j \sqrt{\Id - B_d^* B_d} \right) \right) = \HH.
\end{align*}
\item There exists $m \in \N_0$ such that
\begin{align*}
\bigcap_{j=0}^m \ker \left( \sqrt{\Id - B_d^* B_d} B_d^j \right) = \{ 0 \} \ \ \text{and} \ \ \Span \left( \bigcup_{j=0}^m \im \left( (B_d^*)^j \sqrt{\Id - B_d^* B_d} \right) \right) \text{ is closed}.
\end{align*}
\item There exists $m \in \N_0$ such that
\begin{align}\label{eq:hypocontractivity_condition}
\sum_{j=0}^m (B_d^*)^j (\Id - B_d^* B_d) B_d^j \geq \kappa I
\end{align}
for some $\kappa > 0$.
\end{enumerate}
Moreover, if any of these conditions is satisfied for some $m \in \N_0$, then all conditions are satisfied for this $m$.
\end{lemma}

The following theorem then is the operator analog of Theorem 40 in \cite{AchAM23ELA}.

\begin{theorem}\label{Th:hypocontr}
Let $B_d\in \B(\HH)$ be semi-contractive and $m \in \N_0$. Then the following are equivalent:
\begin{enumerate}
\item $B_d$ has hypocontractivity index $m$.
\item $\Vert B_d^j \Vert = 1$ for $1 \leq j \leq m$ and $\Vert B_d^{m+1} \Vert < 1$.
\end{enumerate}
\end{theorem}
\begin{proof}
We first note that for $x \in \HH$ and $n \in \N_0$ we have
\begin{align*}
\left\langle \left( \sum_{j=0}^n (B_d^*)^j (\Id - B_d^*B_d) B_d^j \right) x, x \right\rangle = \left\langle \left( \Id - (B_d^*)^{n+1} B_d^{n+1} \right) x, x \right\rangle = \Vert x \Vert^2 - \Vert B_d^{n+1} x \Vert^2.
\end{align*}
The statement now follows immediately from the hypocontractivity condition \eqref{eq:hypocontractivity_condition} and the following equivalences for any $0 \leq \kappa < 1$:
\begin{align*}
&\sum_{j=0}^n (B_d^*)^j (\Id - B_d^*B_d) B_d^j \geq \kappa \Id \\
\iff &\text{\upshape for all } x \in \HH: \left\langle \left( \sum_{j=0}^n (B_d^*)^j (\Id - B_d^*B_d) B_d^j \right) x, x \right\rangle \geq \kappa \Vert x \Vert^2 \\
\iff &\text{\upshape for all } x \in \HH: \Vert x \Vert^2 - \Vert B_d^{n+1} x \Vert^2 \geq \kappa \Vert x \Vert^2 \\
\iff &\sqrt{1-\kappa} \geq \Vert B_d^{n+1} \Vert. \qedhere
\end{align*}
\end{proof}
We are now prepared to show the equivalence of the hypocoercivity index and the hypocontractivity index
when the Cayley transform is used for the mapping between the two concepts.

%%%%%%%%%%%%%%%%%%%%%%%%%%%%%%%%%%%%%%%%%%%%%%%%%%%%%
\subsection{Relation between hypocoercivity and hypocontractivity via the Cayley transform}
In this subsection we recall the classical (scaled) Cayley transform (which corresponds to a discretization with the implicit midpoint rule) and show the equivalence of the hypocoercivity index and the hypocontractivity index under this map.

\begin{definition}
Let $\tau > 0$. The \emph{scaled Cayley transform} is the map
\begin{align}\label{scaledcayley}
	M_\tau \colon \C \setminus \left\{ \frac{2}{\tau} \right\} & \longrightarrow \C \setminus\{-1\} \\
	z & \longmapsto \frac{1+\frac{\tau}{2}z}{1-\frac{\tau}{2}z}.
\end{align}
For $\tau=2$, $M = M_2$ is called \emph{Cayley transform}.
\end{definition}

The scaled Cayley transform has the following well-known properties, see e.g. \cite{GreK06}.
\begin{lemma}
Let $\tau > 0$. The scaled Cayley transform \eqref{scaledcayley}
is a bi-holomorphism with inverse given by
\begin{align*}
	M_\tau^{-1} \colon \C \setminus \{ -1 \} & \longrightarrow \C \setminus \left\{ \frac{2}{\tau} \right\} \\
	z & \longmapsto \frac{2(z-1)}{\tau(z+1)}.
\end{align*}
Moreover, $M_\tau(\HHH) = \D$ and $M_\tau^{-1}(\D) = \HHH$.
\end{lemma}

The relationship between the exponential stability of continuous-time and discrete-time systems is given by the following lemma.

\begin{lemma}\label{lem:Cayley}Let $\tau>0$.
\begin{enumerate}
\item If the zero solution of \eqref{linodecc} with $B_c \in \B(\HH)$ is exponentially stable, then the zero solution of $\eqref{linodedd}$ with $B_d \coloneqq M_\tau(-B_c)$ is exponentially stable.
\item If the zero solution of \eqref{linodedd} with $B_d \in \B(\HH)$ is exponentially stable, then the zero solution of $\eqref{linodecc}$ with $-B_c \coloneqq M_\tau^{-1}(B_d)$ is exponentially stable.
\end{enumerate}
\end{lemma}
We remark that $M_\tau(-B_c)$ (resp. $M_\tau^{-1}(B_d)$) are well-defined via the holomorphic functional calculus (see e.g. \cite[Chapter VII, 3, Definition 9]{dunford_schwartz}) if $\frac2\tau\not\in \Lambda(-B_c)$ (resp.~$-1\not\in \Lambda(B_d)$).
\begin{proof}
We prove only (i), the proof for (ii) follows analogously: Let the  zero solution of $\eqref{linodecc}$ be exponentially stable. Then by Theorem \ref{ct_exponentially_stable}, $\Lambda(-B_c) \subset \HHH$. Hence, $B_d \coloneqq M_\tau(-B_c)$ is well-defined, and by the spectral mapping theorem (\cite[Chapter VII, 3, Theorem 11]{dunford_schwartz}) $\Lambda(B_d) = M_\tau(\Lambda(-B_c)) \subset M_\tau(\HHH) = \D$. Consequently, Theorem \ref{dt_exponentially_stable} implies that the zero solution of $\eqref{linodedd}$ is exponentially stable.
\end{proof}
The following well-known fact will be used in our results without explicit mentioning it.
\begin{remark}\label{rem:commute}
For $A \in \B(\HH)$, $\mathcal O \subset \C$ open with $\Lambda(A) \subset \mathcal O$ let $f,g \colon \mathcal O \to \C$ be analytic; hence $f(A)$ commutes with $g(A)$ (\cite[Chapter VII, 3, Theorem 10 b]{dunford_schwartz}). Then rational expressions in $A$ commute with each other if they are well-defined.
\end{remark}

In order to show that not only exponential stability but also the hypocoercivity/hypocontractivity index is preserved under the Cayley transform, we need the following auxiliary lemmas.

\begin{lemma}\label{polynomial_inequality}
Let $A,\tilde A \in \B(\HH)$, $\tilde A\geq 0$ and $n \in \N_0$. Then
\begin{align}\label{eq:polynomial_inequality}
\sum_{k=0}^n (\Id + A^*)^{k} (\Id - A^*)^{n-k} \tilde A (\Id - A)^{n-k} (\Id + A)^{k} \geq \sum_{k=0}^n (A^*)^k \tilde A A^k.
\end{align}
\begin{proof}
We proceed by induction. For $n=0$, \eqref{eq:polynomial_inequality} is trivially satisfied because the left- and right-hand side are equal to $\tilde A$. Now assume that \eqref{eq:polynomial_inequality} holds for some $n \in \N_0$. Then
\begin{align*}
&\sum_{k=0}^{n+1} (\Id + A^*)^{k} (\Id - A^*)^{n+1-k} \tilde A (\Id - A)^{n+1-k} (\Id + A)^{k} \nonumber \\
&= (\Id + A^*)^{n+1} \tilde A (\Id + A)^{n+1} + (\Id - A^*) \left( \sum_{k=0}^n (\Id + A^*)^{k} (\Id - A^*)^{n-k} \tilde A (\Id - A)^{n-k} (\Id + A)^{k} \right) (\Id - A) \nonumber \\
&\geq (\Id + A^*)^{n+1} \tilde A (\Id + A)^{n+1} + (\Id - A^*) \sum_{k=0}^n (A^*)^k \tilde A A^k  (\Id - A).
\end{align*}
Analogously,
\begin{align*}
&\sum_{k=0}^{n+1} (\Id + A^*)^{k} (\Id - A^*)^{n+1-k} \tilde A (\Id - A)^{n+1-k} (\Id + A)^{k} \nonumber \\
&= (\Id - A^*)^{n+1} \tilde A (\Id - A)^{n+1} + (\Id + A^*) \left( \sum_{k=0}^n (\Id + A^*)^{k} (\Id - A^*)^{n-k} \tilde A (\Id - A)^{n-k} (\Id + A)^{k} \right) (\Id + A) \nonumber \\
&\geq (\Id - A^*)^{n+1} \tilde A (\Id - A)^{n+1} + (\Id + A^*) \sum_{k=0}^n (A^*)^k \tilde A A^k  (\Id + A).
\end{align*}
Adding these two inequalities yields
\begin{align*}
&2\sum_{k=0}^{n+1} (\Id + A^*)^{k} (\Id - A^*)^{n+1-k} \tilde A (\Id - A)^{n+1-k} (\Id + A)^{k} \\
&\geq (\Id + A^*)^{n+1} \tilde A (\Id + A)^{n+1} + (\Id - A^*)^{n+1} \tilde A (\Id - A)^{n+1} + 2\sum_{k=1}^{n+1} (A^*)^k \tilde A A^k + 2\sum_{k=0}^n (A^*)^k \tilde A A^k \\
&\geq 2\sum_{k=0}^{n+1} (A^*)^k \tilde A A^k,
\end{align*}
which proves the assertion after dividing by $2$.
\end{proof}
\end{lemma}

\begin{lemma}\label{symmetric_part_discrete_continuous}
Let $A \in \B(\HH)$.
\begin{enumerate}
\item If $-1 \notin \Lambda(A)$, then
\begin{align*}
-\left(M^{-1}(A)\right)_H = (A^* + \Id)^{-1}(\Id - A^* A)(A + \Id)^{-1}.
\end{align*}
\item If $1 \notin \Lambda(A)$, then
\begin{align*}
\Id - M(A)^*M(A) = 4(\Id - A^*)^{-1}(-A_H)(\Id - A)^{-1}.
\end{align*}
\end{enumerate}
\end{lemma}
\begin{proof}
\leavevmode
\begin{enumerate}
\item As $-1 \notin \Lambda(A)$, all expressions are well-defined. Moreover,
\begin{align*}
-\left(M^{-1}(A)\right)_H &= -\frac{1}{2} \left( M^{-1}(A)^* + M^{-1}(A) \right) = -\frac{1}{2} \left( (A^* - \Id)(A^* + \Id)^{-1} + (A - \Id)(A + \Id)^{-1} \right) \\
&= -\frac{1}{2} (A^* + \Id)^{-1} \big( (A^* - \Id)(A + \Id) + (A^* + \Id)(A - \Id) \big) (A + \Id)^{-1} \\
&= (A^* + \Id)^{-1} \left( \Id - A^*A \right) (A + \Id)^{-1}.
\end{align*}
\item Again, as $1 \notin \Lambda(A)$, all expressions are well-defined and
\begin{align*}
\Id - M(A)^*M(A) &= \Id - (\Id + A^*)(\Id - A^*)^{-1}(\Id + A)(\Id - A)^{-1} \\
&= (\Id - A^*)^{-1} \big( (\Id - A^*)(\Id - A) - (\Id + A^*)(\Id + A) \big) (\Id - A)^{-1} \\
&= 4 (\Id - A^*)^{-1} (-A_H) (\Id - A)^{-1}. \qedhere
\end{align*}
\end{enumerate}
\end{proof}

With these two lemmas we can show
the following equivalence result for the hypocoercivity and the hypocontractivity index under the Cayley transform.
For finite-dimensional spaces it was shown in \cite[Theorem 51]{AchAM23ELA} using the second criteria from Lemma \ref{hypocoercivity_equivalent_conditions} and Lemma \ref{hypocontractivity_equivalent_conditions}.
\begin{theorem}\label{index_preserved_by_Cayley}
\leavevmode
\begin{enumerate}
\item If the zero solution of $\eqref{linodecc}$ with $B_c \in \B(\HH)$ is exponentially stable and $B_c$ has hypocoercivity index $m \in \N_0$, then the zero solution of $\eqref{linodedd}$ with $B_d \coloneqq M(-B_c)$ is exponentially stable and $B_d$ has hypocontractivity index $m$.
\item If the zero solution of $\eqref{linodedd}$ with $B_d \in \B(\HH)$ is exponentially stable and $B_d$ has hypocontractivity index $m \in \N_0$, then the zero solution of $\eqref{linodecc}$ with $-B_c \coloneqq M^{-1}(B_d)$ is exponentially stable and $B_c$ has hypocoercivity index $m$.
\end{enumerate}
\end{theorem}
We remark that the well-definedness of $M(-B_c)$ (resp. $M^{-1}(B_d)$) follows just as in Lemma \ref{lem:Cayley}.

\begin{proof}
We proceed by induction on the hypocoercivity/hypocontractivity index $m$ (both at the same time).
For the case that $m=0$ we first assume that $B_c$ has hypocoercivity index $0$, meaning that there exists $\kappa > 0$ such that for the Hermitian part of $B_c$ we have
\begin{align*}
B_H \geq \kappa \Id.
\end{align*}
Using Lemma \ref{symmetric_part_discrete_continuous} (ii), we readily deduce that
\begin{align*}
\Id - M(-B_c)^*M(-B_c) = 4(\Id + B_c^*)^{-1}B_H(\Id +B_c)^{-1} \geq \tilde{\kappa} \Id
\end{align*}
for some $\tilde{\kappa} > 0$ due to  the invertibility of $\Id +B_c$. Consequently, $M(-B_c)$ has hypocontractivity index $0$.

Conversely, assume that $B_d$ has hypocontractivity index $0$, so that there exists $\kappa > 0$ such that
\begin{align*}
\Id - B_d^* B_d \geq \kappa \Id.
\end{align*}
Again, Lemma \ref{symmetric_part_discrete_continuous}(i) yields
\begin{align*}
-\left(M^{-1}(B_d)\right)_H = (B_d^*+\Id)^{-1} (\Id - B_d^* B_d) (B_d+\Id)^{-1} \geq \tilde{\kappa} \Id
\end{align*}
for some $\tilde{\kappa} > 0$ due to the invertibility of $B_d+\Id$. Therefore, $-M^{-1}(B_d)$ has hypocoercivity index $0$.

Suppose now that for some fixed $m \in \N_0$ both statements (i) and (ii) hold for all $m '\in \N_0$ with $m' \leq m$ and assume that $B_c$ has hypocoercivity index $m+1$. Note that the latter assumption rules out the assumptions in statement (i) for $m'\le m$. Using Lemma \ref{symmetric_part_discrete_continuous}(ii) and the definition of the Cayley transform we get
\begin{align}\label{eq:hypocontractivity_expression}
&\sum_{j=0}^{m+1} (M(-B_c)^*)^j \big(\Id - M(-B_c)^* M(-B_c)\big) M(-B_c)^j  \\
&= 4\sum_{j=0}^{m+1} (\Id -B_c^*)^j(\Id +B_c^*)^{-j-1}B_H(\Id +B_c)^{-j-1}(\Id -B_c)^j.\nonumber
\end{align}
Multiplying the right-hand side of \eqref{eq:hypocontractivity_expression} by $\frac{1}{2}(\Id +B_c^*)^{m+2}$ from the left and by $\frac{1}{2}(\Id +B_c)^{m+2}$ from the right yields
\begin{align}\label{eq:hypocontractivity_expression_conjugated}
\sum_{j=0}^{m+1} (\Id -B_c^*)^j(\Id +B_c^*)^{m+1-j}B_H(\Id +B_c)^{m+1-j}(\Id -B_c)^j.
\end{align}
As $\frac{1}{2}(\Id +B_c^*)^{m+2}$ is invertible, it suffices to show that \eqref{eq:hypocontractivity_expression_conjugated} is bounded from below by $\tilde{\kappa} \Id$ for some $\tilde{\kappa} > 0$ in order to deduce that $\eqref{eq:hypocontractivity_expression} \geq  \kappa \Id$ for some $\kappa > 0$. But the former follows immediately by applying Lemma \ref{polynomial_inequality}, which yields
\begin{align*}
\sum_{j=0}^{m+1} (\Id -B_c^*)^j(\Id +B_c^*)^{m+1-j}B_H(\Id +B_c)^{m+1-j}(\Id -B_c)^j \geq \sum_{j=0}^{m+1} (-B_c^*)^j B_H (-B_c)^j
\end{align*}
and the right-hand side is bounded from below by $\tilde{\kappa} \Id$ for some $\tilde{\kappa} > 0$ by assumption. On the one hand this shows that the hypocontractivity index of $M(-B_c)$ is at most $m+1$. But on the other hand it cannot be strictly smaller than $m+1$ due to the induction hypothesis, which in particular says that the hypocontractivity index of $M(-B_c)$ is smaller than $m+1$ if and only if the hypocoercivity index of $B_c$ is smaller than $m+1$.

The corresponding statement starting with the assumption that the hypocontractivity index of $B_d$ is $m+1$ can be proven analogously.
%\end{itemize}
\end{proof}

\begin{remark}\label{rem:diffproof}
The proof in the infinite-dimensional setting differs from the finite-dimensional proof in \cite[Theorem 51]{AchAM23ELA} due to the following reasons: First, the closedness in part (ii) of Lemma~\ref{hypocoercivity_equivalent_conditions} respectively Lemma~\ref{hypocontractivity_equivalent_conditions} is automatically satisfied in the finite-dimensional case and therefore not needed there. Thus, the proof of \cite[Theorem 51]{AchAM23ELA} does not include an argument why the closedness is preserved by the Cayley transform. Moreover, the square root in parts (i) and (ii) of these lemmas can be left out in finite-dimensional settings (see (B1') and (D1) in \cite{AchAM23ELA}), which allows one to use the relation between $\Id-B_d^*B_d$ and $B_H$ given by Lemma \ref{symmetric_part_discrete_continuous}. Note that this relation does not carry over to the square root terms because our operators are not assumed to be normal (and hence $A$, $A^*$ do not commute in general). Therefore, we have rather used Criterion (iii) from Lemma \ref{hypocoercivity_equivalent_conditions} and Lemma \ref{hypocontractivity_equivalent_conditions} to prove Theorem \ref{index_preserved_by_Cayley} as there is no square root appearing there, and thus we can use Lemma \ref{symmetric_part_discrete_continuous}.
\end{remark}

Next we extend Theorem \ref{index_preserved_by_Cayley} to the scaled Cayley transform. This will be used for the time discretization of continuous-time systems via the implicit midpoint rule.
\begin{theorem}\label{index_preserved_by_scaled_Cayley}
Let $B_c$, respectively $B_d$, be in $\B(\HH)$ and suppose that $\tau>0$ is such that $M_\tau(-B_c)$, respectively $M_\tau^{-1}(B_d)$, are well defined in $\B(\HH)$.
\begin{enumerate}
\item If the zero solution of $\eqref{linodecc}$ is exponentially stable and $B_c$ has hypocoercivity index $m \in \N_0$, then the zero solution of $\eqref{linodedd}$ with $B_d \coloneqq M_\tau(-B_c)$ is exponentially stable and $B_d$ has hypocontractivity index $m$.
\item If the zero solution of $\eqref{linodedd}$ is exponentially stable and $B_d$ has hypocontractivity index $m \in \N_0$, then the zero solution of $\eqref{linodecc}$ with $-B_c \coloneqq M_\tau^{-1}(B_d)$ is exponentially stable and $B_c$ has hypocoercivity index $m$.
\end{enumerate}
\end{theorem}
\begin{proof}
We first note that $B_c$ and $\tau B_c$ have the same hypocoercivity index. This follows immediately from part (i) of Lemma \ref{hypocoercivity_equivalent_conditions} because for any $j \in \N_0$ the operators $(B_c^*)^j\sqrt{B_H}$ and $(\tau B_c^*)^j\sqrt{\tau B_H}$
have the same image. Consequently, the first assertion follows immediately from Theorem \ref{index_preserved_by_Cayley}(i) and
\begin{align*}
m_{dHC}(M_\tau(-B_c)) = m_{dHC}\left(M\left(-\frac{\tau}{2}B_c\right)\right) = m_{HC}\left(\frac{\tau}{2}B_c\right) = m_{HC}(B_c),
\end{align*}
where $m_{HC}$ denotes the hypocoercivity and $m_{dHC}$ the hypocontractivity index.

For the second assertion one easily verifies that $M\circ M_\tau^{-1}=M_{4/\tau}\circ M^{-1}$, and hence
\[
  B_c:=-M_\tau^{-1}(B_d) = -M^{-1}(M_{4/\tau}( M^{-1}(B_d))).
\]
Using, in this order, Theorem \ref{index_preserved_by_Cayley} (ii), Theorem \ref{index_preserved_by_scaled_Cayley} (i), and again Theorem \ref{index_preserved_by_Cayley} (ii) we obtain the claim:
\begin{align*}
  m &= m_{dHC}(B_d) = m_{HC}\big(-M^{-1}(B_d)\big) = m_{dHC}\big(M_{4/\tau}( M^{-1}(B_d))\big)  \\
  &= m_{HC}\big(-M^{-1}(M_{4/\tau}( M^{-1}(B_d)))\big). \qedhere
\end{align*}
\end{proof}

At the end of this subsection we discuss the connection of the decay rates for the continuous and discrete dynamics, i.e. for $B_c$ and $B_d \coloneqq M_{\tau}(-B_c)$, respectively. The following results are direct consequences of the spectral mapping theorem for the holomorphic functional calculus (\cite[Chapter VII, 3, Theorem 11]{dunford_schwartz}) and the mapping properties of the transformations $M_\tau$ and $M_\tau^{-1}$.

\begin{proposition}\label{decayconnection}
Let $\tau>0$ and $B_c\in\B(\HH)$ satisfy $-\frac2\tau\notin\Lambda(B_c)$, and set $B_d\coloneqq M_\tau(-B_c)$.
Furthermore, for $z \in \C$ and $r > 0$ let $K_r(z) \coloneqq \{ w \in \C: \lvert z - w \rvert < r \}$.
\begin{enumerate}
\item For some fixed $r \in (0,1)$ let $R_1 \coloneqq \frac{4}{\tau} \frac{r}{1-r^2}$ and $z_0 \coloneqq \frac{2}{\tau} \frac{1+r^2}{1-r^2}$. Then $\Lambda(B_c) \subset K_{R_1}(z_0)$ if and only if $\Lambda(B_d) \subset K_r(0)$. In this case, $\rho(B_d) < r$ and $\lambda=r$ is a decay rate in \eqref{discrete-decay}.
Note that the above assumption on $\Lambda(B_c)$ implies for the spectral abscissa $\mu(-B_c) < -\frac{2}{\tau} \frac{1-r}{1+r} \eqqcolon -R_2$.
% spectral gap $s(B_c) > \frac{2}{\tau} \frac{1-r}{1+r} \eqqcolon R_2$.
\item For some fixed $r \in (0,1)$ let $R_3 \coloneqq \frac{1+r}{2}$ and $z_1 \coloneqq \frac{r-1}{2}$. Then $\Lambda(B_d) \subset K_{R_3}(z_1)$ if and only if $\mu(-B_c) < -R_2$. % $s(B_c) > R_2$.
In this case, $\lambda=R_2$ is an admissible decay rate in \eqref{def:HC}.
\end{enumerate}
\end{proposition}
The above two mappings are illustrated in Figure \ref{fig:discretization-long-term-decay}.

Part (i) of the above proposition implies: If $\Lambda(B_c)$ is unbounded and if $B_c$ gives rise to a suitable spectral mapping theorem, then $\Lambda(B_d) \not\subset K_r(0)$  for any $r < 1$, hence $\rho(B_d) \geq 1$ and $B_d$ is not hypocontractive. More concretely, the explicit form of $M_\tau$ shows that $-1=\lim_{\lvert z \rvert \to \infty} M_\tau(z) \in \Lambda(B_d)$. This reflects the fact that the implicit midpoint rule, although A-stable, is not L-stable,  \cite{HaiW96}.

\begin{figure}[H]
    \centering
    \resizebox{0.8\textwidth}{!}{%
    \begin{tikzpicture}[
        >=stealth,
        font=\small,
        line width=0.55pt,
        spectral axis/.style={black,->},
        spectral boundary/.style={red,dashed,line width=0.8pt},
        spectral hatch/.style={pattern=north east lines,pattern color=red},
        spectral tick/.style={black,line width=0.55pt},
        spectral radius/.style={black,<->,line width=0.55pt}
    ]
    % Editable parameters, shared by all four panels.
    \pgfmathsetmacro{\cayleyr}{1/2}
    \pgfmathsetmacro{\cayleytau}{2}
    % Derived quantities from the spectral-mapping proposition.
    \pgfmathsetmacro{\cayleyRone}{4/\cayleytau*\cayleyr/(1-\cayleyr^2)}
    \pgfmathsetmacro{\cayleyzZero}{2/\cayleytau*(1+\cayleyr^2)/(1-\cayleyr^2)}
    \pgfmathsetmacro{\cayleyRtwo}{2/\cayleytau*(1-\cayleyr)/(1+\cayleyr)}
    \pgfmathsetmacro{\cayleyRthree}{(1+\cayleyr)/2}
    \pgfmathsetmacro{\cayleyzOne}{(\cayleyr-1)/2}

    % Upper left: the open disk mapping to the disk of radius r.
    \begin{scope}[shift={(-1,0)}]
        \path[spectral hatch] (\cayleyzZero,0) circle[radius=\cayleyRone];
        \draw[spectral boundary] (\cayleyzZero,0) circle[radius=\cayleyRone];
        \draw[spectral axis] (-0.6,0) -- (3.65,0) node[below right] {$\mathrm{Re}$};
        \draw[spectral axis] (0,-1.55) -- (0,1.55) node[above left] {$\mathrm{Im}$};
        \draw[spectral radius] (\cayleyzZero,0) -- ++(60:\cayleyRone)
            node[midway,above left,fill=white,inner sep=1pt] {$R_1$};
        \draw[spectral tick] (\cayleyRtwo,-0.045) -- (\cayleyRtwo,0.045)
            node[below=3pt,fill=white,inner sep=1pt] {$R_2$};
        \fill[black] (\cayleyzZero,0) circle[radius=1pt];
        \node[below=5pt,fill=white,inner sep=1pt] at (\cayleyzZero,0) {$z_0$};
        \draw[spectral tick] ({\cayleyzZero+\cayleyRone},-0.045)
            -- ({\cayleyzZero+\cayleyRone},0.045)
            node[below=3pt,fill=white,inner sep=1pt] {$z_0+R_1$};
        \node[font=\normalsize] at (1.35,-2.05) {$\Lambda(B_c)$};
    \end{scope}

    % Upper right: the open disk of radius r inside the unit circle.
    \begin{scope}[shift={(7.55,0)}]
        \path[spectral hatch] (0,0) circle[radius=\cayleyr];
        \draw[black,line width=0.8pt] (0,0) circle[radius=1];
        \draw[spectral boundary] (0,0) circle[radius=\cayleyr];
        \draw[spectral axis] (-1.55,0) -- (1.55,0) node[below right] {$\mathrm{Re}$};
        \draw[spectral axis] (0,-1.55) -- (0,1.55) node[above left] {$\mathrm{Im}$};
        \draw[spectral tick] (-1,-0.045) -- (-1,0.045) node[below=3pt,fill=white,inner sep=1pt] {$-1$};
        \draw[spectral tick] (1,-0.045) -- (1,0.045) node[below=3pt,fill=white,inner sep=1pt] {$1$};
        \draw[spectral tick] (-\cayleyr,-0.045) -- (-\cayleyr,0.045)
            node[above=3pt,fill=white,inner sep=1pt] {$-r$};
        \draw[spectral tick] (\cayleyr,-0.045) -- (\cayleyr,0.045)
            node[above=3pt,fill=white,inner sep=1pt] {$r$};
        \node[font=\normalsize] at (0,-2.05) {$\Lambda(B_d)$};
    \end{scope}

    \draw[->,black,very thick] (2.95,1.65)
        .. controls (3.7,2.25) and (5.05,2.25) .. (5.8,1.65)
        node[midway,above=3pt] {$M_\tau(-z)$};
    % Lower left: the open disk tangent to the unit circle at -1.
    \begin{scope}[shift={(0,-4.8)}]
        \path[spectral hatch] (\cayleyzOne,0) circle[radius=\cayleyRthree];
        \draw[black,line width=0.8pt] (0,0) circle[radius=1];
        \draw[spectral boundary] (\cayleyzOne,0) circle[radius=\cayleyRthree];
        \draw[spectral axis] (-1.55,0) -- (1.55,0) node[below right] {$\mathrm{Re}$};
        \draw[spectral axis] (0,-1.55) -- (0,1.55) node[above left] {$\mathrm{Im}$};
        \draw[spectral radius] (\cayleyzOne,0) -- ++(130:\cayleyRthree)
            node[midway,above right,fill=white,inner sep=1pt] {$R_3$};
        \draw[spectral tick] (-1,-0.045) -- (-1,0.045) node[below=3pt,fill=white,inner sep=1pt] {$-1$};
        \draw[spectral tick] (1,-0.045) -- (1,0.045) node[below=3pt,fill=white,inner sep=1pt] {$1$};
        \draw[spectral tick] (\cayleyr,-0.045) -- (\cayleyr,0.045)
            node[above=3pt,fill=white,inner sep=1pt] {$r$};
        \fill[black] (\cayleyzOne,0) circle[radius=1pt];
        \node[below=5pt,fill=white,inner sep=1pt] at (\cayleyzOne,0) {$z_1$};
        \node[font=\normalsize] at (0,-2.05) {$\Lambda(B_d)$};
    \end{scope}

    % Lower right: the open half-plane; only its boundary is drawn.
    \begin{scope}[shift={(6.2,-4.8)}]
        \begin{scope}
            \clip (-0.6,-1.5) rectangle (3.45,1.5);
            \path[spectral hatch] (\cayleyRtwo,-2) rectangle (4,2);
        \end{scope}
        \draw[spectral boundary] (\cayleyRtwo,-1.5) -- (\cayleyRtwo,1.5);
        \draw[spectral axis] (-0.6,0) -- (3.65,0) node[below right] {$\mathrm{Re}$};
        \draw[spectral axis] (0,-1.55) -- (0,1.55) node[above left] {$\mathrm{Im}$};
        \draw[spectral tick] (\cayleyRtwo,-0.045) -- (\cayleyRtwo,0.045)
            node[below right=3pt,fill=white,inner sep=1pt] {$R_2$};
        \node[font=\normalsize] at (1.35,-2.05) {$\Lambda(B_c)$};
    \end{scope}

    \draw[->,black,very thick] (1.8,-3.15)
        .. controls (2.7,-2.25) and (4.4,-2.25) .. (5.3,-3.15)
        node[midway,above=3pt] {$-M_\tau^{-1}(z)$};

    \end{tikzpicture}%
    }
    \caption{Location of the spectrum of $B_c$ and $B_d$ when they are related by the midpoint rule. }
    \label{fig:discretization-long-term-decay}
\end{figure}

In this section we have shown that the Cayley transformation maps between the continuous and the discrete setting and leads to an equivalence between the hypocoercivity index and the hypocontractivity index. In the next section we analyze the discretization of the continuous-time solution operator and the discrete approximation of the short-time decay.

%%%%%%%%%%%%%%%%%%%%%%%%%%%%%%%%%%%%%%%%%%%%%%%%%%%%%%%%%%%%%%%%%%%%%%

\section{Hypocontractivity as finite difference approximation of the short-time decay in hypocoercive systems}\label{sec:t-discrete}

In this section we show that the implicit midpoint rule
recovers the first non-vanishing Peano derivative of the
continuous-time propagator norm at $t=0$ through a
normalized finite difference as $\tau\to0+$.
The order $2m_{HC}+1$ of this derivative can be arbitrarily
high, although the midpoint rule is only second-order
convergent. Together with the preservation of the
hypocoercivity/hypocontractivity index, this is a structural asymptotic
property not implied by the usual convergence estimate.

Consider the  implicit midpoint rule  with time step $\tau$ (the Crank-Nicolson scheme), see e.g. \cite{PisZ07}, for the operator equation
\eqref{linodecc}
 at discrete time steps $t_k = k \tau $ with step size $\tau > 0$.
With $x_k \approx x(t_k)$ this scheme reads
\[ \frac{x_{k+1} - x_k}{\tau} = -B_c\left(\frac{x_{k+1} + x_k}{2}\right) \]
which after rearrangement gives
\begin{equation}
    \label{eqq1}x_{k+1} = B_d x_k
\end{equation}
where $B_d = (\Id+\frac{\tau}{2}B_c)^{-1} (\Id-\frac{\tau}{2} B_c) = M_\tau(-B_c)$
is the scaled Cayley transform of $-B_c$. While $\tau$ was arbitrary but fixed in the previous section, $B_d=B_d(\tau)$ will here be considered as a function of $\tau>0$.

As has been established in Section \ref{sec:short-time}, if the hypocoercivity index of $B_c$ is $m_{HC} \in \N_0$ then the propagator norm of \eqref{linodecc} satisfies
\begin{equation}\label{exp-decay2}
  \Phi(t) \coleq \norm{e^{-B_ct}} = 1 - c t^{2m_{HC}+1} + o(t^{2m_{HC}+1}) \quad \text{\upshape for }t\to0+
\end{equation}
with the constant $c$ given explicitly by \eqref{c-operator}.
In finite dimensional systems, the propagator norm $\Phi(t)$ is real analytic on some interval $[0,\delta)$ (see \cite{KOHAUPT20011}), but for infinite dimensional $\HH$ this is not true in general: \cite{Graf} presents an example, where $\Phi$ is not even $C^1$ on any interval $[0,\delta)$. The short-time decay formula \eqref{exp-decay2} shows that $\Phi$ is $(2m_{HC}+1)$ times \emph{Peano differentiable} at $t=0$ (this is essentially its definition, see \cite{Fis08}). Moreover, $\Phi$ is (Fr\'echet) differentiable at $t=0$, but not necessarily continuously differentiable on any interval $[0,\delta)$. An example of an $o(t^3)$-function that is differentiable at $t=0$ but not $C^1$ on any $[0,\delta)$ is given by $f(t):=t^4 \sin(1/t^3)$.
The constant $c$ in \eqref{exp-decay2} is related to the Peano derivative of order $2m_{HC}+1$ at $t=0$, $\Phi_{(2m_{HC}+1)}(0)$ as
\[
c = -\lim_{t \to 0+} \frac{\norm{e^{-B_ct}}-1}{t^{2m_{HC}+1}} =: - \frac{\Phi_{(2m_{HC}+1)}(0)}{(2m_{HC}+1)!} \,.
\]
For finite dimensional systems, the local analyticity of $\Phi$ implies that all Peano derivatives equal their classical counterparts (also in Theorem \ref{th:FinDiff-limit} below).

Next we introduce a discrete analog of the propagator norm $\Phi(t)$ from \eqref{exp-decay2}: For any fixed step size $\tau>0$ we define the grid %discrete
function
\[
\phi(\tau)=\big(\phi_k(\tau);\, k\in\N_0\big) \quad \text{\upshape with }
  \phi_k(\tau):= \|B_d(\tau)^k\|.
\]
Note that, for $\tau$ small, for any \emph{fixed number of steps} $k\in\N$, $\phi_k(\tau)$
is (at least) a third order approximation for $\Phi(t_k)$
due to the convergence order of the midpoint rule.

Inspired by the odd symmetry about $(0,1)$ of the dominant terms in \eqref{exp-decay2}, i.e.\ $1-ct^{2m_{HC}+1}$, we extend $\phi(\tau)$ to negative values of $k$:
\[
  \phi_k(\tau):= 2- \phi_{-k}(\tau),\quad k\in -\N.
\]
This construction is illustrated in Figures \ref{fig:approx_index1} and \ref{fig:approx_index2}, respectively, for the hypocoercive matrices
\[
  B_c:=\begin{bmatrix} 0 & 1 \\ -1 & 1 \end{bmatrix}\quad \text{with } m_{HC}=1,\quad \text{and }\quad B_c:=\begin{bmatrix} 0 & 1 & 0 \\ -1 & 0 & 1 \\ 0 & -1 & 1 \end{bmatrix} \quad \text{with } m_{HC}=2,
\]
and $\tau=\frac12$ in both cases.

\begin{figure}
  \centering
  \includegraphics[width=0.8\textwidth]
  {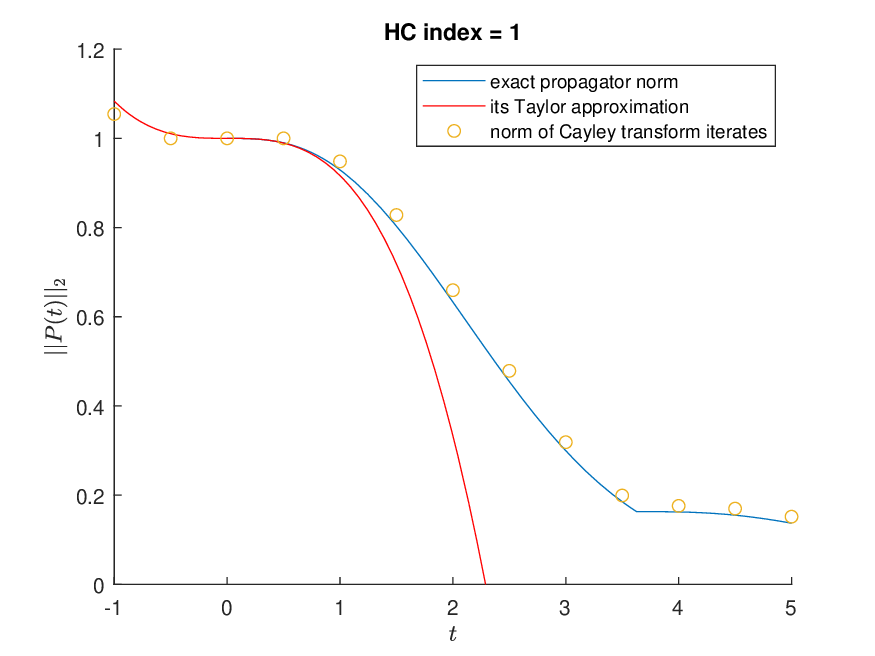}
  \caption{Example with HC-index 1: Approximation of the propagator norm $\|e^{-B_ct}\|$ (blue) by its 3rd order Taylor expansion about $t=0$ (red), and the norm of iterates of the scaled Cayley transform, $\|B_d(\tau)^k\|$ with step size $\tau=\frac12$ (circles). Note the extension of the latter to negative indices, using odd symmetry. This is needed for constructing the symmetric finite difference of the grid function $\phi(\tau)$, centered at $k=0$. }
  \label{fig:approx_index1}
\end{figure}

\begin{figure}
  \centering
  \includegraphics[width=0.8\textwidth]
  {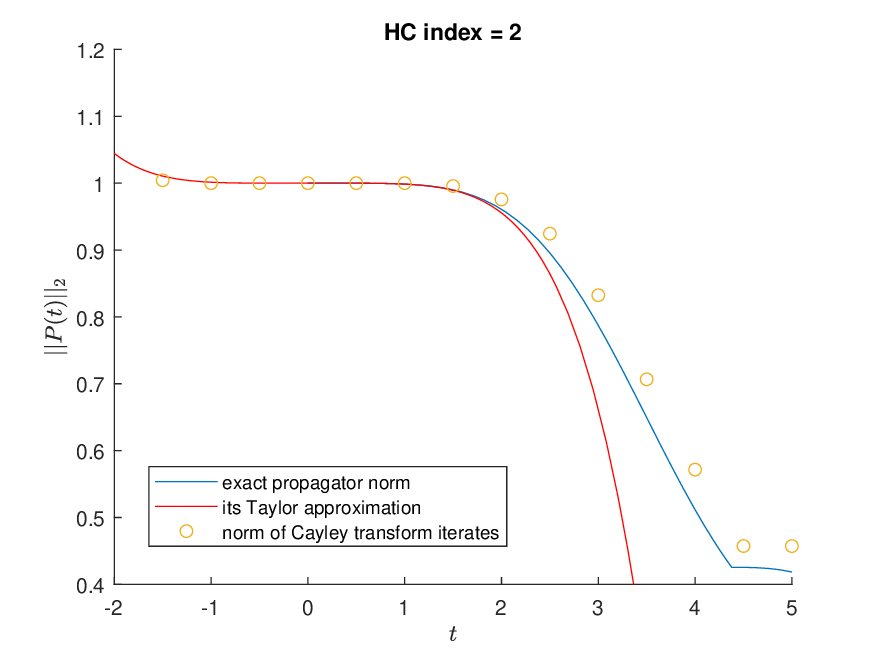}
  \caption{Example with HC-index 2: Approximation of the propagator norm $\|e^{-B_ct}\|$ (blue) by its 5th order Taylor expansion about $t=0$ (red), and the norm of iterates of the scaled Cayley transform, $\|B_d(\tau)^k\|$ with step size $\tau=\frac12$ (circles). Note the extension of the latter to negative indices, using odd symmetry. This is needed for constructing the symmetric finite difference of the grid function $\phi(\tau)$, centered at $k=0$. }
  \label{fig:approx_index2}
\end{figure}

For a matrix $B_c$ with some $m_{HC}\in\N_0$ and the corresponding scaled Cayley transform $B_d(\tau)$ with some fixed $\tau>0$, we consider next the symmetric finite difference of order $2m_{HC}+1$ of the discrete function $\phi(\tau)$ about the index $k=0$. Since $\|B_d^j\|=1$ for $j=0,...,m_{HC}$ (see Theorem \ref{Th:hypocontr}(ii)), the above mentioned finite difference satisfies
\[
  \left[\Delta^{2m_{HC}+1}\phi(\tau)\right]_{k=0} =
  -\frac12 \phi_{-m_{HC}-1}(\tau) + \frac12 \phi_{m_{HC}+1}(\tau)
  =\phi_{m_{HC}+1}(\tau)-1.
\]
Here we have used that the coefficients in the finite difference of order $2m_{HC}+1$ are skew-symmetric about the index $k=0$, the numerical stencil is of length $2m_{HC}+3$, and the first/last coefficients are always $\mp\frac12$. E.g., the coefficients for the first derivative are $-\frac12,\,0,\,\frac12$, and for the third derivative $-\frac12,\,1,\,0,\,-1,\,\frac12$, see \cite{For88}.

In this section we prove the following connection between the symmetric central difference of $(2m_{HC}+1)$st order of $\phi(\tau)$ and the corresponding Peano derivative of the propagator norm $\Phi$ at $t=0$.

\begin{theorem}\label{th:FinDiff-limit}
Let $-B_c\in\B(\HH)$ be semi-dissipative, and let $m_{HC}\in\N_0$ be the HC-index of $B_c$. Let $B_d(\tau):=M_\tau(-B_c),\,\tau>0$ be the corresponding family of scaled Cayley transforms. Then
\begin{equation}\label{FinDiff-limit}
\frac{\left[\Delta^{2m_{HC}+1}\phi(\tau)\right]_{k=0}}{\tau^{2m_{HC}+1}} = \frac{\phi_{m_{HC}+1}(\tau) - 1}{\tau^{2m_{HC}+1}} \: \stackrel{\tau\to0+}{\longrightarrow} \:\Phi_{(2m_{HC}+1)}(0) = - (2m_{HC}+1)!\,c,
\end{equation}
with the constant $c$ given by \eqref{c-operator}.
Moreover,
\begin{equation}\label{FinDiff-limit2}
 \left[\Delta^{j}\phi(\tau)\right]_{k=0} = \Phi_{(j)}(0) = \delta_0^j \quad \text{for } j=0,...,2m_{HC}.
\end{equation}
\end{theorem}
\begin{proof}
The proof will be given in Subsection~\ref{sec:FinDiff-limit}.
\end{proof}

Theorem~\ref{th:FinDiff-limit} implies that the norm of the first strictly contractive power of the discrete propagator from \eqref{eqq1}, $B_d(\tau)^{m_{HC}+1}$, has an asymptotic expansion similar to \eqref{exp-decay2}.

\begin{corollary}\label{cor:FinDiff-limit}
Under the assumptions of Theorem \ref{th:FinDiff-limit} it holds that
\begin{equation}\label{eqqq1}
    \phi_{m_{HC}+1}(\tau) \coleq \norm{B_d(\tau)^{m_{HC}+1}} = 1 - (2m_{HC}+1)! \,c \tau^{2m_{HC}+1} + o(\tau^{2m_{HC}+1}) \quad \text{for }\tau\to0+
\end{equation}
with the constant $c$ again given by \eqref{c-operator}.
\end{corollary}

Note, however, the additional factor $(2m_{HC}+1)!$ in the expansion~\eqref{eqqq1}. This is due to the difference in interpretation: \eqref{exp-decay2} is an expansion in time, while \eqref{eqqq1} is an expansion in the step size.
Due to \eqref{eqqq1} and \eqref{exp-decay2}, $\phi_{m_{HC}+1}(\tau)$ is an approximation of order $2m_{HC}+1$ of
\[
\Phi(t_{m_{HC}+1})=1-(m_{HC}+1)^{2m_{HC}+1} c \tau^{2m_{HC}+1} + o(\tau^{2m_{HC}+1}),
\]
while the midpoint rule would in general only yield third order here.

%%%%%%%%%%%%%%%%%%%%%%%%%%%%%%%%%%%%%%%%%%%
\subsection{An illustrating example}

In the following example we illustrate the theoretical results from Corollary \ref{thm_propagator_norm_asymptotics_finite_dim} and Theorem \ref{th:FinDiff-limit} with a concrete matrix.

\begin{example}
Let $\HH = \C^2$, and consider the matrix
\begin{align*}
B_c :=
\begin{bmatrix}
0 &\frac{1}{2} \\
-\frac{1}{2} &1
\end{bmatrix},
\end{align*}
which has the double eigenvalue $\frac12$.
Its Hermitian part
\begin{align*}
B_H =
\begin{bmatrix}
0 &0 \\
0 &1
\end{bmatrix} \ge0
\end{align*}
is not positive definite. But
\begin{align*}
B_H + B_c^* B_H B_c = \begin{bmatrix}
\frac{1}{4} &-\frac{1}{2} \\
-\frac{1}{2} &2
\end{bmatrix}
\end{align*}
is positive definite
because its eigenvalues are strictly positive. Consequently, $B_c$ is hypocoercive with hypocoercivity index $m_{HC}=1$. In view of Corollary  \ref{thm_propagator_norm_asymptotics_finite_dim} we want to compute the asymptotic behavior of $\Vert e^{-tB_c} \Vert$ for $t\to 0+$. To this end, we first note that the matrix exponential is explicitly given by
\begin{align*}
e^{-tB_c} = e^{-\frac{t}{2}} \begin{bmatrix}
1 + \frac{t}{2} &-\frac{t}{2} \\
\frac{t}{2} &1-\frac{t}{2}
\end{bmatrix},
\end{align*}
which follows from \cite[Corollary 2.3 i)]{Bernstein1993SomeEF}. Moreover, the spectral norm of a real valued matrix $\begin{bmatrix} a &b \\ c &d \end{bmatrix}$ is explicitly given by
\begin{align}\label{eq:explicit_spectral_norm}
\sqrt{\frac{g + \sqrt{g^2 - 4h}}{2}},
\end{align}
where $g = a^2 + b^2 + c^2 + d^2$ and $h = (ad-bc)^2$, which follows from the fact that the spectral norm is given by the largest singular value. Applying this formula to $e^{\frac{t}{2}}e^{-tB_c}$ we deduce $g = 2+t^2$ and $h=1$ and so
\begin{align*}
\Phi(t):=
\left\Vert e^{-tB_c} \right\Vert = e^{-\frac{t}{2}}\sqrt{\frac{2 + t^2 + t \sqrt{4+t^2}}{2}}.
\end{align*}
As this expression is analytic for $t \geq 0$, we can obtain the asymptotic behavior for $t \to 0+$ by differentiation which yields
\begin{align*}
\left\Vert e^{-tB_c} \right\Vert = 1 - \frac{1}{48}t^3 + \mathcal{O}(t^4)
\end{align*}
and $\Phi_{(3)}(0)=\Phi^{(3)}(0)=-\frac18$.
As
\begin{align*}
\min_{\substack{y \in \ker(\sqrt{B_H})\\ \Vert y \Vert = 1}} \left\Vert \sqrt{B_H} B_c y \right\Vert^2 = \left\langle B_c^* B_H B_c \begin{bmatrix}1\\0\end{bmatrix}, \begin{bmatrix}1\\0\end{bmatrix} \right\rangle = \frac{1}{4},
\end{align*}
we obtain exactly the constant given in Corollary \ref{thm_propagator_norm_asymptotics_finite_dim}.

Next, we consider the corresponding operator for the discrete-time system with time step $\tau > 0$:
\begin{align*}
B_d(\tau) = M_\tau(-B_c) =
\frac{1}{(4+\tau)^2}
\begin{bmatrix}
16 + 8\tau - \tau^2
&
-8\tau
\\
8\tau
&
16-8\tau-\tau^2
\end{bmatrix}.
\end{align*}
Again, we use \eqref{eq:explicit_spectral_norm} to deduce the spectral norm of $(4+\tau)^2 B_d(\tau)$. In this case,
\begin{align*}
g &= 2(16-\tau^2)^2 + 256 \tau^2, \\
h &= (16-\tau^2)^4,
\end{align*}
and furthermore
\begin{align*}
\sqrt{g^2-4h} &= 32 \tau (\tau^2+16), \\
g + \sqrt{g^2-4h} &= 2(\tau+4)^4.
\end{align*}
Hence,
\begin{align*}
\left\Vert \begin{bmatrix}
16 + 8\tau - \tau^2
&
-8\tau
\\
8\tau
&
16-8\tau-\tau^2
\end{bmatrix} \right\Vert = (4+\tau)^2
\end{align*}
and thus
\begin{align*}
\Vert B_d(\tau) \Vert = 1
\end{align*}
for all $\tau > 0$, as was expected by Theorem \ref{index_preserved_by_Cayley}(i) and Theorem \ref{Th:hypocontr}(ii). Next, we consider
\begin{align*}
B_d(\tau)^2 =
\frac{1}{(4+\tau)^4}
\begin{bmatrix}
(\tau^2 - 8\tau - 16)^2 - 64\tau^2 &
16\tau(\tau^2 - 16) \\
-16\tau(\tau^2 - 16) &
(\tau^2 + 8\tau - 16)^2 - 64\tau^2
\end{bmatrix}.
\end{align*}
As before, we use \eqref{eq:explicit_spectral_norm} to deduce the spectral norm of $(4+\tau)^4 B_d(\tau)^2$: We have
\begin{align*}
g &= 2(\tau^2-16)^2 (\tau^4 + 480\tau^2 + 256),  \\
h &= (\tau-4)^8(\tau+4)^8,
\end{align*}
and furthermore
\begin{align*}
g^2 - 4h &= 4096 \tau^2 (\tau^2-16)^4 (\tau^4 + 224 \tau^2 + 256), \\
\frac{g + \sqrt{g^2-4h}}{2} &= (\tau-4)^2(\tau+4)^2 \left( \tau^4 + 480\tau^2 + 256 + 32\tau \sqrt{\tau^4 + 224 \tau^2 + 256} \right).
\end{align*}
Thus,
\begin{align}\label{Bd2-norm}
\phi_2(\tau):=
\left\Vert B_d(\tau)^2 \right\Vert = \frac{|4-\tau|}{(\tau + 4)^3}  \sqrt{ \tau^4 + 480\tau^2 + 256 + 32\tau \sqrt{\tau^4 + 224 \tau^2 + 256} }.
\end{align}
Note that for $\tau=4$, the matrix $B_d$ is singular and $B_d^2$ is the zero matrix (which is connected to $\Lambda(B_c)=\{\frac12\}$ and \eqref{scaledcayley}). Then \eqref{Bd2-norm} is analytic for $0\leq\tau <4$, and we can obtain the asymptotic behavior for $\tau \to 0+$ by differentiation which yields
\begin{align*}
\left\Vert B_d(\tau)^2 \right\Vert = 1 - \frac{1}{8}\tau^3 + \mathcal{O}(\tau^4).
\end{align*}
Hence,
\[
  \frac{\phi_2(\tau)-1}{\tau^3}\stackrel{\tau\to0+}{\longrightarrow} \Phi_{(3)}(0)=-\frac18,
\]
which is exactly the asymptotic expansion predicted by Theorem \ref{th:FinDiff-limit}.
\end{example}

%%%%%%%%%%%%%%%%%%%%%%%%%%%%%%%%%%%%%%
\subsection{Proof of Theorem \ref{th:FinDiff-limit}}
\label{sec:FinDiff-limit}

In this subsection we  prove Theorem \ref{th:FinDiff-limit}. The equality \eqref{FinDiff-limit2} follows trivially from Theorem \ref{Th:hypocontr}(ii) and \eqref{exp-decay2}. Hence we prove here \eqref{eqqq1} which is equivalent to \eqref{FinDiff-limit}.

The proofs follow those of Section \ref{sec:short-time}, the main difference being that instead of $e^{-B_c^*t}e^{-B_ct}$, which has been expanded as the Cauchy product of two series, we have to expand $(B_d^*)^{\mHC+1}B_d^{\mHC+1}$ here, which slightly complicates things.

In the following, we will tacitly assume that $\tau$ is small enough so that $B_d(\tau)$ is well-defined; e.g., $\tau < 2/\norm{B_c}$ is sufficient for this. In this case, $I \pm \frac\tau 2B_c$ and $(I \mp \frac \tau 2 B_c)^{-1}$ commute, see Remark \ref{rem:commute}. Moreover, in the following lemmata we formulate the assumptions mostly in terms of $B_c$ and not for $B_d(\tau)$, since the former is just one single operator and the proofs often involve $B_c$, anyhow.

\begin{lemma}\label{mylem5.1} For any $m \in \N_0$,
\begin{equation}\label{formel}
(B_d^*(\tau))^{m+1}B_d(\tau)^{m+1} = I -2\tau \sum_{j=0}^m (B_d^*(\tau))^j \left(I+\frac \tau 2 B_c\right)^{-*} B_H \left(I+\frac \tau 2 B_c\right)^{-1} B_d(\tau)^j.
\end{equation}
\end{lemma}
\begin{proof}
We show the claim by induction. For $m=0$ we have
\begin{align*}
    B^*_d B_d &= \left(I - \frac\tau 2 B_c\right)^* \left(I + \frac \tau 2 B_c\right)^{-*} \left(I + \frac \tau 2 B_c\right)^{-1} \left(I - \frac \tau 2 B_c\right) \\
    &= \left(I + \frac \tau 2 B_c\right)^{-*} \left(I - \frac \tau 2 B_c\right)^* \left(I - \frac \tau 2 B_c\right)\left(I + \frac \tau 2 B_c\right)^{-1} \\
    & = I + \left(I + \frac \tau 2 B_c\right)^{-*} \left[ \left( I - \frac \tau 2 B_c\right)^* \left(I - \frac \tau 2 B_c \right) - \left( I + \frac \tau 2 B_c\right)^* \left(I + \frac \tau 2 B_c \right) \right] \left(I + \frac \tau 2 B_c\right)^{-1} \\
    & = I + \left(I + \frac \tau 2 B_c\right)^{-*} \left( - \frac \tau 2 \left( B_c^* + B_c \right) \cdot 2 \right) \left(I + \frac \tau 2 B_c\right)^{-1} \\
    & = I - 2\tau \left( I + \frac \tau 2 B_c\right)^{-*} B_H \left(I + \frac \tau 2 B_c\right)^{-1}.
\end{align*}
Suppose the claim holds for some $m \in \N_0$. Then we have
\begin{align*}
    (B_d^*)^{m+2} B_d^{m+2} &= B_d^* \left( I - 2\tau \sum_{j=0}^m (B_d^*)^j \left(I + \frac \tau 2 B_c\right)^{-*} B_H \left(I + \frac \tau 2 B_c\right)^{-1} B_d^j \right) B_d \\
    & = I - 2 \tau \left(I + \frac \tau 2 B_c\right)^{-*} B_H \left(I + \frac \tau 2 B_c \right)^{-1} \\
    & \qquad - 2 \tau \sum_{j=1}^{m+1} (B_d^*)^j \left(I + \frac \tau 2 B_c\right)^{-*} B_H \left( I + \frac \tau 2 B_c\right)^{-1} B_d^j
\end{align*}
which is the statement for $m+1$.
\end{proof}
Below we will need the series expansions
\begin{align*}
    \left(I + \frac \tau 2 B_c\right)^{-1} &= \sum_{i=0}^\infty \left(-\frac 1 2 \right)^i \tau^i B_c^i, \\
    \left(I + \frac \tau 2 B_c\right)^{-*} &= \sum_{i=0}^\infty \left(-\frac 1 2 \right)^i \tau^i (B_c^*)^i,  \\
    B_d(\tau) &= \left(I + \frac \tau 2 B_c\right)^{-1}\left(I - \frac \tau 2 B_c \right) \\
    & = \sum_{i=0}^\infty \left(-\frac 1 2\right)^i \tau^i B_c^i + \sum_{i=0}^\infty \left(-\frac 1 2 \right)^{i+1} \tau^{i+1} B_c^{i+1} \\
    &= I - \sum_{i=1}^\infty \left(-\frac 1 2 \right)^{i-1} \tau^i B_c^i = \sum_{i=0}^\infty a_i \tau^i B_c^i
\end{align*}
with
\begin{equation}\label{a-def}
  a_i := \begin{cases} 1 & i=0, \\  2\left(-\frac 1 2\right)^{i} & i \ge 1. \end{cases}
\end{equation}

In the following, for any given $j \in \N$ we abbreviate $p = (p_1, \dotsc, p_j) \in \N_0^j$, $\abso{p} = p_1 + \dotsc + p_j$ and $a_p = a_{p_1} \dotsm a_{p_j}$ for $a_i$ given in \eqref{a-def}, and similarly for $q \in \N_0^j$. For $j=0$ we set $\abso{p}=0$ and $a_p = 1$.

\begin{lemma}
\label{mylem4.1}
Let $m\in\N_0$. Then
\begin{equation} \label{infformula}\inf_{\substack{\lambda_1,\dotsc,\lambda_m \in \C \\ \lambda_0=1}} \sum_{j=0}^m \sum_{\abso{p}+i=0}^m \sum_{\abso{q}+l =0}^m \lambda_{m-\abso{q}-l} \overline{\lambda_{m-\abso{p} - i}} a_p a_q \left(-\tfrac 12\right)^{i+l} = \binom{2m}{m}^{-1}, \end{equation}
and the minimum is attained at a unique $\lambda^* =(\lambda_1^*,...,\lambda_m^*) \in\R^m$.
\end{lemma}

\begin{remark}\label{rem:sum}
Note that for $m=0$ the infimum is omitted  in \eqref{infformula}.
We also remark that the second sum in \eqref{infformula} is actually a $(j+1)$-fold summation, over the $j$ components of the vector $p$ and the scalar index $i\in\N_0$. So the sum over the indices $\abso{p}+i=0$ to $m$ equals the sum $p_1 + \dotsc + p_j + i = 0$ to $m$.
The same holds for the third sum.

These sums arise from the Cauchy product of the series representation of the operators in \eqref{formel}, namely $j$ factors of $B_d^*$, the factor $(I+\frac\tau 2 B_c)^{-*}$, the factor $(I+\frac \tau 2 B_c)^{-1}$ and $j$ factors of $B_d$, whose contribution to the coefficients in the $(2j+2)$-fold Cauchy product are enumerated by the indices $p_1, \dotsc, p_j$, $i$, $l$, and $q_1 ,\dotsc, q_j$, respectively.
\end{remark}

We have the following proof of Lemma~\ref{mylem4.1}:
\begin{proof}
Since the claim is trivial for $m=0$, we assume $m\ge1$ in the sequel.
We first note that the expression to be minimized equals $\norm{T}^2 = \sum_{j=0}^m \abso{T_j}^2$ with
\begin{equation}\label{Tj}
  T_j = T_j(\lambda) := \sum_{\abso{p}+i=0}^m \lambda_{m-\abso{p}-i} a_p \left(-\tfrac 12\right)^i,\quad j = 0 ,\dotsc, m,
\end{equation}
where $\lambda =(\lambda_1,...,\lambda_m) \in\C^m$ and $\lambda_0=1$.
The generating functions of the sequences $(a_i)_{i\in\N_0}$ and $((-1/2)^i)_{i\in\N_0}$ are
\[ F(z) = \sum_{i=0}^\infty a_i z^i = \frac{1-z/2}{1+z/2} \quad\text{and}\quad G(z) = \sum_{i=0}^\infty \left(-\tfrac 12\right)^i z^i = \frac{1}{1+z/2} \qquad (z \in \C,\, |z|<2), \]
respectively. With
\begin{align}
    W^{(j)}(z) &\coleq F(z)^j G(z) = \frac{\left(1-z/2\right)^j}{\left(1+z/2\right)^{j+1}} = \sum_{r=0}^\infty w_r^{(j)} z^r, \label{Wejoterr} \\
    w_r^{(j)} & :=\sum_{\abso{p}+i=r} a_p \Big(-\frac12\Big)^i, \quad j = 0 ,\dotsc, m,\label{wejoterr}
\end{align}
and $L(z) \coleq \sum_{k=0}^m \lambda_k z^k$
we see that $T_j = [z^m]\big(L(z) W^{(j)}(z)\big)$, where $[z^m]f(z)$ denotes the coefficient of $z^m$ in the Taylor series of $f(z)$. For every fixed $j\in\{0,...,m\}$ we have:
\begin{align*}
    [z^m] \left( L(z) W^{(j)}(z) \right) &= [z^m] \sum_{k=0}^m \lambda_k z^k \left(\sum_{i=0}^\infty a_i z^i\right)^j \sum_{i=0}^\infty \left(- \frac 1 2 \right)^i z^i \\
    &= [z^m] \sum_{k=0}^m \lambda_k z^k \sum_{p_1=0}^\infty a_{p_1} z^{p_1} \dotsm \sum_{p_j=0}^\infty a_{p_j} z^{p_j} \sum_{i=0}^\infty \left(-\frac 1 2\right)^i z^i \\
    &= \sum_{k+p_1 + \dotsc + p_j + i = m} \lambda_k a_{p_1} \dotsm a_{p_j} \left(-\frac 12\right)^i = \sum_{\abso{p}+i=0}^m \lambda_{m-\abso{p}-i} a_p \left(-\frac 1 2\right)^i = T_j.
\end{align*}

With $\alpha_j \coleq (-1)^j \binom{m}{j}$ for $j=0 ,\dotsc, m$ it follows that
\begin{equation}\label{eq10}
\begin{aligned}
    \sum_{j=0}^m \alpha_j W^{(j)}(z) & = \frac{1}{1+z/2} \sum_{j=0}^m (-1)^j \binom{m}{j} \left( \frac{1-z/2}{1+z/2} \right)^j \\
    & = \frac{1}{1+z/2} \left(1-\frac{1-z/2}{1+z/2}\right)^m = \frac{z^m}{(1+z/2)^{m+1}} \\
    & = z^m \sum_{i=0}^\infty \binom{-m-1}{i} \left(\frac z 2\right)^i = z^m \sum_{i=0}^\infty \left(-\frac 1 2 \right)^i \binom{m+i}{i} z^i \\
    &= z^m \left( 1 - \frac z 2 + \frac{z^2}{4} - \dotsc \right)^{m+1} .
\end{aligned}
\end{equation}
This gives
\begin{equation}\label{oben}
\begin{aligned}
\sum_{j=0}^m \alpha_j T_j &= \sum_{j=0}^m \alpha_j [z^m](L(z) W^{(j)}(z)) = [z^m] \sum_{j=0}^m \alpha_j L(z) W^{(j)}(z) \\
&= [z^m] \left( L(z) \frac{z^m}{(1+z/2)^{m+1}} \right) = \lambda_0 = 1.
\end{aligned}
\end{equation}
By the Cauchy-Schwarz inequality,
\begin{equation}\label{CS-inf}
  \norm{T}^2 = \sum_{j=0}^m \abso{T_j}^2 \ge \frac{\left(\sum_{j=0}^m \alpha_j T_j \right)^2}{\sum_{j=0}^m \alpha_j^2} = \left(\sum_{j=0}^m \binom{m}{j}^2\right)^{-1} = \binom{2m}{m}^{-1},
\end{equation}
and the r.h.s.\ is the claimed infimum from \eqref{infformula}. So it remains to show that this infimum is attained.

Equality in \eqref{CS-inf} holds if $(T_0, \dotsc, T_m)$ is proportional to $(\alpha_0, \dotsc, \alpha_m)$.
This proportionality can be attained for a suitable choice of $\lambda = (\lambda_1, \dotsc, \lambda_m) \in\R^m$.
To this end we rewrite \eqref{Tj} as
\[ T_j = \sum_{r=0}^m \left( \sum_{\abso{p}+i=r} a_p \Big(-\frac12\Big)^i  \right) \lambda_{m-r} = \sum_{r=0}^m w_r^{(j)}\lambda_{m-r} = \sum_{k=1}^m Q_{j,k} \lambda_k + b_j, \]
with $p\in\N_0^j$, the matrix $Q_{j,k} := w_{m-k}^{(j)}$, and the vector $b_j := w_m^{(j)}$ for $j=0 ,\dotsc, m$ and $k=1,\dotsc, m$. The quadratic form in $\lambda$,
\[ \norm{T}^2 = \norm{Q\lambda+b}^2 \]
has a unique minimizer $\lambda^*\in\R^m$ because the columns of the matrix $Q\in\R^{m+1,m}$ are linearly independent.

To see this, suppose that $Q\beta = 0$ for some $\beta = (\beta_1, \dotsc, \beta_m)\in\R^m$, i.e.,
\[ \sum_{r=0}^{m-1} \beta_{m-r} w_r^{(j)} = 0 \quad \text{for } j=0 ,\dotsc, m. \]
If we show that $q_r(j) := w_r^{(j)}$ is a polynomial in $j$ of degree $r$ (defined via \eqref{wejoterr} at the grid points $j=0,...,m$), $\beta=0$ will follow because the $q_r$ then are linearly independent polynomials.
The polynomial structure of $q_r(j)$ is seen from the following combinatorial argument: For each fixed $r\in\{0,...,m-1\}$, the highest number of identical contributions in the sum of \eqref{wejoterr} is obtained when $i=0$ and if the vector $p\in\N_0^j$ has exactly $r$ times the entry $1$ and $j-r$ times the entry $0$.  There are
\[
  \binom{j}{r} = \frac{j(j-1)...(j-r+1)}{r!},
\]
(i.e.\ a polynomial of degree $r$ in $j$) terms of this sort, with the value $a_1^r a_0^{j-r} (-\frac12)^0$. All other terms in \eqref{wejoterr} occur in a lower power of $j$.

Hence, $Q$ indeed has full column rank.

At the minimum of $\norm{Q \lambda + b}^2$ the gradient $2Q^T(Q\lambda^*+b)$ vanishes, which means that $T^* := Q\lambda^*+b\in\R^{m+1}$ is in the one-dimensional subspace orthogonal to the columns of $Q$. From \eqref{eq10} and \eqref{Wejoterr} we obtain
\begin{align}\label{eq11}
    \sum_{j=0}^m \alpha_j w^{(j)}_r &= 0 \qquad \text{for }0 \le r \le m-1,
\end{align}
which means that $\alpha\in\R^{m+1}$ is orthogonal to the columns of $Q$ and hence proportional to $T^*$.
Hence $T^*$ makes \eqref{CS-inf} an equality, and the corresponding $\lambda^*$ minimizes \eqref{infformula}.
\end{proof}

\begin{lemma}
\label{mylem4.2}
Let $m \in \N_0$ and let $y \colon \R^+ \to \HH$ satisfy $\norm{y(\tau)}=\Theta(1)$ and
\begin{equation}\label{eq1}
    \norm{\sqrt{B_H}B_c^n y(\tau)} = o(\tau^{m-1-n})
\end{equation}
as $\tau \to 0+$ for $0 \le n < m$.(For $m=0$ condition \eqref{eq1} is void.) We also note that $y$ does not even have to be continuous. Then
\begin{multline}
\left\lVert B_d(\tau)^{m+1} \frac{y(\tau)}{\norm{y(\tau)}}\right\rVert^2 = 1 - \frac{2\tau}{\norm{y(\tau)}^2} \sum_{j=0}^m \sum_{\abso{p}+i=0}^m \sum_{\abso{q}+l=0}^m \tau^{\abso{p}+i+\abso{q}+l} a_p a_q \left(-\tfrac 1 2\right)^{i+l} \label{eq5}\\
\nonumber\cdot \left\langle \sqrt{B_H} B_c^{\abso{q} + l} y(\tau), \sqrt{B_H} B_c^{\abso{p} + i} y(\tau) \right\rangle + o(\tau^{2m+1})
\end{multline}
for $\tau \to 0+$.
Just as in Lemma \ref{mylem4.1}, the second sum above is actually a $(j+1)$-fold summation, over the $j$ components of the vector $p$ and the scalar index $i\in\N_0$. The same holds for the third sum.
\end{lemma}

\begin{proof}
We use Lemma \ref{mylem5.1} and the series expansions mentioned thereafter to write
\begin{equation}\label{BByy}
  \left\langle (B_d^*(\tau))^{m+1}B_d(\tau)^{m+1} \frac{y(\tau)}{\norm{y(\tau)}}, \frac{y(\tau)}{\norm{y(\tau)}} \right\rangle
\end{equation}
as the Cauchy product of the respective series, resulting in
\begin{gather}
\nonumber   \eqref{BByy} = 1 - \frac{2\tau}{\norm{y(\tau)}^2} \sum_{j=0}^m \sum_{r=0}^{2m} \sum_{\abso{p} + i + l + \abso{q} = r} a_p a_q \left(-\tfrac 1 2\right)^{i+l} \left\langle (B_c^*)^{\abso{p} + i} B_H B_c^{\abso{q} + l} y(\tau), y(\tau) \right\rangle \tau^r + \cO(\tau^{2m+2}) \\
\label{eq2}    = 1 - \frac{2\tau}{\norm{y(\tau)}^2} \sum_{j=0}^m \sum_{r=0}^{2m} \sum_{\abso{p} + i + l + \abso{q} = r} a_p a_q \left(-\tfrac 1 2\right)^{i+l} \left\langle \sqrt{B_H} B_c^{\abso{q} + l} y(\tau), \sqrt{B_H} B_c^{\abso{p} + i} y(\tau) \right\rangle \tau^r + \cO(\tau^{2m+2}).
\end{gather}
In the above sum, consider now for every $r=0,\dotsc ,2m$ the term
\[ \tau^{r+1} \left\langle \sqrt{B_H} B_c^{\abso{q} + l} y(\tau), \sqrt{B_H} B_c^{\abso{p} + i} y(\tau) \right\rangle. \]
By \eqref{eq1}, if $\abso{q}+l > m$ (and hence $\abso{p}+i < m$) this term is of order $\tau^{r+1}\cO(1) o(\tau^{m-1-\abso{p}-i}) = o(\tau^{m+\abso{q}+l}) = o(\tau^{2m+1})$, and similarly if $\abso{p}+i > m$. Hence, \eqref{eq2} becomes
\begin{multline*}
1 - \frac{2\tau}{\norm{y(\tau)}^2} \sum_{j=0}^m \sum_{\abso{p}+i=0}^m \sum_{\abso{q}+l=0}^m \tau^{\abso{p}+i+\abso{q}+l} a_p a_q \left(-\tfrac 1 2\right)^{i+l} \\
\cdot \left\langle \sqrt{B_H} B_c^{\abso{q} + l} y(\tau), \sqrt{B_H} B_c^{\abso{p} + i} y(\tau) \right\rangle + o(\tau^{2m+1}).
\end{multline*}
\end{proof}

The following theorem yields the (as it turns out) optimal lower bound for the asymptotic expansion of $\norm{ B_d(\tau)^{m+1}}$, the norm of the scaled Cayley transform, as a function of the step size and in the limit  $\tau\to 0+$. Thus it is the discrete analogue of Theorem \ref{operator_norm_lower_bound}. Again, we essentially assume that the HC-index of $B_c$ is larger or equal to some (fixed) integer $m$.
\begin{theorem}\label{mythm4.3}
Let $-B_c\in \mathcal B(\HH)$ be semi-dissipative, and let $B_H$ be the Hermitian part of $B_c$.
Suppose that, for some $m \in \N_0$,  the coercivity condition
\begin{align}\label{eq:coercivity_assumption2}
    \sum_{j=0}^{m-1} (B_c^*)^j B_H B_c^j \geq \kappa I
\end{align}
is \emph{violated} for all $\kappa > 0$. Then
\[ \norm{ B_d(\tau)^{m+1}}^2 \ge 1 - \tau^{2m+1} \frac{2}{\binom{2m}{m}} \lim_{\delta \to 0+} \inf_{\substack{\norm{\sqrt{B_H}B_c^py}\le \delta,\\0 \le p < m,\\ \norm{y}=1}} \norm{\sqrt{B_H}B_c^my }^2 + o(\tau^{2m+1}) \]
for $\tau \to 0+$.

We recall from Theorem \ref{operator_norm_lower_bound} that, for $m=0$ the infimum is just taken over $\SS$.
\end{theorem}
\begin{proof}
As in Theorem \ref{operator_norm_lower_bound} we set
\[ S_\tau \coleq \{ y  \in \SS: \norm{\sqrt{B_H}B_c^p y} \le \tau^{m+1-p},\ 0 \le p < m \} \]
and choose for each $\tau>0$ some $y_0(\tau) \in S_\tau$ with
\[ \left\vert \norm{\sqrt{B_H} B_c^m y_0(\tau)} - \inf_{y \in S_\tau} \norm{\sqrt{B_H}B_c^m y} \right\vert < \tau, \]
resulting in
\begin{align}
\label{do1}    \norm{\sqrt{B_H}B_c^p y_0(\tau)} &= \cO(\tau^{m+1-p}) & & (0 \le p < m)  \,, \\
\label{do2}    \norm{\sqrt{B_H}B_c^p y_0(\tau)} &= \cO(\tau^{m-p}) & & (p \in \N_0).
\end{align}
From \eqref{eq:infimum_convergence} we recall the identity
\[ \lim_{\tau \to 0+} \norm{\sqrt{B_H} B^m_c y_0(\tau)} = \lim_{\tau \to 0+} \inf_{y \in S_\tau} \norm{\sqrt{B_H}B_c^m y} = \lim_{\delta \to 0+} \inf_{\substack{\norm{\sqrt{B_H}B_c^py}\le \delta,\\0 \le p < m,\\ \norm{y}=1}} \norm{\sqrt{B_H}B_c^my }. \]
We now use the ansatz
\[ y(\tau) \coleq y_0(\tau) + \sum_{s=1}^m \lambda_s(\tau B_c)^s y_0(\tau) = \sum_{s=0}^m \lambda_s(\tau B_c)^s y_0(\tau),\quad \tau > 0, \]
where $\lambda_0=1$, and $\lambda_1, \dotsc, \lambda_m\in\R$ are to be chosen later. Then $\Vert y(\tau) \Vert = \Theta(1)$. Moreover, equations \eqref{do1} and \eqref{do2} translate into
\begin{equation}\label{do3}
    \norm{\sqrt{B_H}B_c^n(\tau B_c)^s y_0(\tau)} = \cO(\tau^{m+1-n})
\end{equation}
for $0 \le s,n$ with $s+n<m$; and
\begin{equation}\label{do4}
    \norm{\sqrt{B_H}B_c^n(\tau B_c)^s y_0(\tau)} = \cO(\tau^{m-n})
\end{equation}
for any $s,n \in \N_0$. Hence, $y(\tau)$ satisfies all conditions of Lemma \ref{mylem4.2} and we obtain
\begin{align}
    \nonumber\left\lVert B_d(\tau)^{m+1} \frac{y(\tau)}{\norm{y(\tau)}} \right\rVert^2 & = 1 - \frac{2\tau}{\norm{y(\tau)}^2} \sum_{j=0}^m \sum_{\abso{p}+i=0}^m \sum_{\abso{q}+l=0}^m \tau^{\abso{p}+i+\abso{q}+l} a_p a_q \left(-\frac 1 2\right)^{i+l} \\
    \label{eq3} & \cdot \left\langle \sqrt{B_H} B_c^{\abso{q} + l} y(\tau), \sqrt{B_H} B_c^{\abso{p} + i} y(\tau) \right\rangle+ o(\tau^{2m+1}).
\end{align}

To expand the inner product
\[ \sum_{r=0}^m \sum_{s=0}^m \lambda_r \lambda_s \left\langle \sqrt{B_H} B_c^{\abso{q} + l} (\tau B_c)^r y_0(\tau), \sqrt{B_H} B_c^{\abso{p} + i} (\tau B_c)^s y_0(\tau) \right\rangle \]
we distinguish the following cases:

\underline{$r+\abso{q}+l<m$:} The above estimates \eqref{do3}, \eqref{do4} in the form
\begin{align*}
\norm{\sqrt{B_H}B_c^{\abso{q} + l}(\tau B_c)^r y_0(\tau)} &= \cO(\tau^{m+1-\abso{q}-l}) \\
\norm{\sqrt{B_H}B_c^{\abso{p} + i}(\tau B_c)^s y_0(\tau)} &= \cO(\tau^{m-\abso{p}-i})
\end{align*}
yield
\[ \tau^{1+\abso{p}+i+\abso{q}+l} \left\langle \sqrt{B_H} B_c^{\abso{q} + l} (\tau B_c)^r y_0(\tau), \sqrt{B_H} B_c^{\abso{p} + i} (\tau B_c)^s y_0(\tau) \right\rangle = \cO(\tau^{2m+2}). \]

\underline{$s+\abso{p}+i<m$:} The same estimate holds.

\underline{$r+\abso{q}+l \ge m+1$ and $s+\abso{p}+i \ge m$} or \underline{$r+\abso{q}+l \ge m$ and $s+\abso{p}+i \ge m+1$:}
\begin{multline*}
    \tau^{1+\abso{p}+i+\abso{q}+l} \left\langle \sqrt{B_H} B_c^{\abso{q} + l} (\tau B_c)^r y_0(\tau), \sqrt{B_H} B_c^{\abso{p} + i} (\tau B_c)^s y_0(\tau) \right\rangle \\
    = \tau^{1+s+\abso{p}+i+r+\abso{q}+l} \left\langle \sqrt{B_H} B_c^{\abso{q} + l+r} y_0(\tau), \sqrt{B_H} B_c^{\abso{p} + i+s} y_0(\tau) \right\rangle = \cO(\tau^{2m+2})
\end{multline*}
using $\norm{\sqrt{B_H}B_c^{r+\abso{q}+l}y_0(\tau)} = \cO(1)$ and $\norm{\sqrt{B_H}B_c^{s+\abso{p}+i}y_0(\tau)} = \cO(1)$.

Hence, we only need to retain the terms with $s+\abso{p}+i = r + \abso{q}+l = m$ in the estimate of \eqref{eq3}, and it simplifies as follows:
\begin{align*}
    \left\lVert B_d(\tau)^{m+1} \frac{y(\tau)}{\norm{y(\tau)}} \right\rVert^2 & = 1 - \frac{2\tau}{\norm{y(\tau)}^2} \sum_{j=0}^m \sum_{\abso{p}+i=0}^m \sum_{\abso{q}+l=0}^m \sum_{r=0}^m \sum_{s=0}^m \tau^{\abso{p}+i+\abso{q}+l+r+s} \lambda_r {\lambda_s} a_p a_q \left(-\frac 1 2\right)^{i+l} \\
    & \cdot \left\langle \sqrt{B_H} B_c^{\abso{q} + l +r } y_0(\tau), \sqrt{B_H} B_c^{\abso{p} + i +s } y_0(\tau) \right\rangle + o(\tau^{2m+1}) \\
    &= 1 - 2 \tau^{2m+1} \sum_{j=0}^m \sum_{\abso{p}+i=0}^m \sum_{\abso{q}+l=0}^m\lambda_{m-\abso{q}-l} {\lambda_{m-\abso{p}-i}} a_p a_q \left(-\frac 1 2\right)^{i+l} \\
    & \cdot \norm{\sqrt{B_H}B_c^my_0(\tau)}^2 + o(\tau^{2m+1}),
\end{align*}
where we used $\norm{y(\tau)}^{-2} = 1 + \cO(\tau)$.
Choose $\lambda_1, \dotsc, \lambda_m$ according to Lemma \ref{mylem4.1} such that the expression
\[ \sum_{j=0}^m \sum_{\abso{p}+i=0}^m \sum_{\abso{q}+l=0}^m\lambda_{m-\abso{q}-l} \lambda_{m-\abso{p}-i} a_p a_q \left(-\frac 1 2\right)^{i+l} \]
is minimal and thus equals $\binom{2m}{m}^{-1}$. Then,
\[ \left\lVert B_d(\tau)^{m+1} \frac{y(\tau)}{\norm{y(\tau)}} \right\rVert^2 = 1 - \tau^{2m+1} \frac{2}{\binom{2m}{m}} \norm{\sqrt{B_H} B_c^m y_0(\tau)}^2 + o(\tau^{2m+1}), \]
and hence using \eqref{eq:infimum_convergence}:
\[ \norm{ B_d(\tau)^{m+1}}^2 \ge \norm{ B_d(\tau)^{m+1} \frac{y(\tau)}{\norm{y(\tau)}} }^2 = 1 - \tau^{2m+1} \frac{2}{\binom{2m}{m}} \lim_{\delta \to 0+} \inf_{\substack{\norm{\sqrt{B_H}B_c^py}\le \delta,\\0 \le p < m,\\ \norm{y}=1}} \norm{\sqrt{B_H}B_c^my }^2 + o(\tau^{2m+1}). \qedhere \]
\end{proof}

\begin{remark}
\label{myrem4.4}
If $B_c$ has hypocoercivity index greater than or equal to $m \in \N_0$, then $\norm{B_d(\tau)^{m+1}}= 1 - \cO(\tau^{2m+1})$ by Theorem \ref{mythm4.3}.
\end{remark}

The following two lemmas are technical ingredients for proving  the optimal upper bound on $\|B_d(\tau)^{m+1}\|$ in Theorem \ref{mythm4.7}.

\begin{lemma}
\label{mylem4.5}
Let $k \in \N_0$. Then, for every $d = 0 ,\dotsc, k$
it holds
\begin{equation}\label{mylem4.5.eq}
C^{(k)}_d \coleq \inf_{\substack{z \in \HH^{k+1}\\ z_d \ne 0}} \sum_{j=0}^k \sum_{\abso{p} + i = 0}^k  \sum_{\abso{q} + l = 0}^k a_p a_q \left(-\tfrac{1}{2}\right)^{i+l} \left\langle z_{|q| + l}, z_{|p| + i}\right\rangle \norm{z_d}^{-2} > 0,
\end{equation}
where $z = (z_0, \dotsc, z_k )\in \HH^{k+1}$ and $p,\,q\in\N_0^j$.
Moreover, we have $C_k^{(k)} = \binom{2k}{k}^{-1}$.
\end{lemma}

\begin{proof}
First, we abbreviate for each $j=0 ,\dotsc, k$ and $z=(z_0,\dotsc,z_k) \in \HH^{k+1}$
\[
 M_j(z) := \sum_{\abso{p} + i = 0}^k \sum_{\abso{q} + l = 0}^k a_p a_q \left(-\tfrac{1}{2}\right)^{i+l} \left\langle z_{|q| + l}, z_{|p| + i}\right\rangle = \norm{\sum_{\abso{p}+i=0}^k a_p (-\tfrac12)^i  z_{\abso{p}+i}}^2,
\]
which shows that the expression on the left-hand side of \eqref{mylem4.5.eq} is real and nonnegative. Note that, for any unit vector $e \in \HH$ and $z_r^\parallel \coleq \langle z_r, e \rangle e$, $r = 0, \dotsc, k$, the inequality
\begin{equation*}
\sum_{j=0}^k M_j(z_0, \dotsc, z_k) \ge \sum_{j=0}^k M_j(z_0^\parallel, \dotsc, z_k^\parallel) \end{equation*}
holds as a general property of orthogonal projections, here onto $\sspan[e]$. Hence, the infimum of the left-hand side of \eqref{mylem4.5.eq} w.r.t.\ $z\in\HH^{k+1}$ with $z_d \ne 0$ equals the infimum over all collinear vectors $z_0,\dotsc, z_k$ with $z_d \ne 0$. Fix some $d\in\{0,\dots,k\}$, suppose that $z_d \ne 0$, and set $e := z_d/\norm{z_d}$, which implies that $\norm{z_d^\parallel} = \norm{z_d}$. Then we can restrict the minimization in \eqref{mylem4.5.eq} to the scalar case $z_r := \mu_r e$ with $\mu_r \in \C$, $r=0,...,k$ and write
\[ \sum_{j=0}^k M_j (z_0, \dotsc, z_k) = \sum_{j=0}^k \left|\sum_{r=0}^k w_r^{(j)} \mu_r\right|^2 = Q(\mu), \]
where $w^{(j)}_r$ was given in \eqref{wejoterr}, $\mu = (\mu_0, \dotsc, \mu_k)^T$, and $Q$ is the quadratic form given by $Q(\mu) = \norm{W\mu}^2$ with the matrix $W = (w^{(j)}_r)_{j,r=0 ,\dotsc, k}$. Because the entries of $W$ are obtained by evaluating the $k+1$ linearly independent polynomials $q_r(j) = w^{(j)}_r$, $r= 0, \dots, k$ (as functions of $j$) at the $k+1$ different points $j=0,\dots,k$ (see the proof of Lemma \ref{mylem4.1}), $W$ is invertible. Hence $Q$ is a positive definite quadratic form. It follows that
\[ \sum_{j=0}^k M_j(z) \norm{z_d}^{-2} = Q(\mu) \abso{\mu_d}^{-2} \]
and we obtain
\begin{equation}\label{Cdk}
    C^{(k)}_d = \inf_{\mu \in \C^{k+1}, \abso{\mu_d}=1} Q(\mu) \ge \lambda_{\min}(W^*W) > 0,
\end{equation}
where $\lambda_{\min}(W^*W)$ denotes the smallest eigenvalue of $W^*W$. This shows that all $C^{(k)}_d$ are positive.

We recall the definition of $T_j$ from \eqref{Tj}, with $m$ replaced now by $k$:
\[
  T_j(\lambda) := \sum_{\abso{p}+i=0}^k \lambda_{k-\abso{p}-i} a_p \left(-\tfrac 12\right)^i,\quad j = 0 ,\dotsc, k,
\]
where $\lambda =(\lambda_1,...,\lambda_k) \in\C^k$ and $\lambda_0=1$. At $\lambda=0$ this gives with \eqref{wejoterr} $T_j(0) = \sum_{|p|+i=k} a_p(-\tfrac12)^i=w_k^{(j)}; \: j = 0 ,\dotsc, k$.
Then, for $d=k$ and $\alpha_j := (-1)^j \binom k j$, $j = 0,\dots,k$ as in the proof of Lemma \ref{mylem4.1} we know from \eqref{eq11} and \eqref{oben} that $W^T \alpha = e_k$ (the $k$-th column of the identity matrix on $\C^{k+1}$ when starting to count with column 0). Hence
\begin{equation}\label{Q-lower-bound}
    \abso{\mu_k}^2 = \abso{\alpha^T W \mu}^2 \le \norm{\alpha}^2 Q(\mu) \implies Q(\mu) \ge \frac{\abso{\mu_k}^2}{\norm{\alpha}^2} \quad \forall \mu \in \C^{k+1}.
\end{equation}
Setting $\tilde\mu := W^{-1}\alpha$ we have $\abso{\tilde\mu_k}^2 = \norm{\alpha}^4$ and hence $Q(\tilde\mu)=\norm{\alpha}^2 = \abso{\tilde\mu_k}^2/\norm{\alpha}^2$. Thus the lower bound from \eqref{Q-lower-bound} is attained. Scaling $\tilde\mu$ by $\tilde\mu_k$, \eqref{Cdk} yields $C^{(k)}_k = \norm{\alpha}^{-2} = \binom{2k}{k}^{-1}$, where we have used \eqref{CS-inf}.
\end{proof}

\begin{lemma}
\label{mylem4.6}
If $B_c$ has hypocoercivity index greater than or equal to $m \in \N_0$ and $y \colon \R^+ \to \HH$ satisfies $\norm{y(\tau)}=1$ for $\tau > 0$ and $\norm{B_d(\tau)^{m+1}y(\tau)} - \norm{B_d(\tau)^{m+1}} = o(\tau^{2m+1})$ for $\tau\to0+$. Then $\norm{\sqrt{B_H}B_c^p y(\tau)} = o(\tau^{m-1-p})$ for $0 \le p < m$.
\end{lemma}

\begin{proof}
We will show inductively for $k=0,1,\dotsc,m$ that
\[ \norm{\sqrt{B_H}B_c^n y(\tau)} = o(\tau^{k-1-n}) \quad\text{for}\quad 0 \le n < k. \]
For $k=0$ there is nothing to show. For the step from $k<m$ to $k+1$ we use $\norm{B_d(\tau)} \le 1$, Lemma \ref{mylem4.2} (with $m$ there replaced by $k$ here), and Lemma \ref{mylem4.5} to obtain
\begin{align*}
    &\norm{B_d(\tau)^{m+1}y(\tau)}^2  \le \norm{B_d(\tau)^{k+1}y(\tau)}^2 \\
    & \quad = 1 - 2\tau \sum_{j=0}^k \sum_{\abso{p}+i=0}^k \sum_{\abso{q}+l=0}^k a_p a_q \Big(-\frac12\Big)^{i+l} \left\langle \tau^{\abso{q}+l}\sqrt{B_H}B_c^{\abso{q}+l}y(\tau), \tau^{\abso{p}+i}\sqrt{B_H}B_c^{\abso{p}+i}y(\tau) \right\rangle + o(\tau^{2k+1}) \\
    & \quad \le 1 - 2\tau C^{(k)}_n \norm{\tau^n \sqrt{B_H} B_c^n y(\tau)}^2 + o(\tau^{2k+1}) \quad \forall n = 0 ,\dots, k.
\end{align*}
This implies that
\begin{align*}
    2\tau^{2n+1} C^{(k)}_n \norm{\sqrt{B_H}B_c^ny(\tau)}^2 & \le 1 - \norm{B_d(\tau)^{m+1}y(\tau)}^2 + o (\tau^{2k+1}) \\
    & = 1 - \norm{B_d(\tau)^{m+1}}^2 - \left( \norm{B_d(\tau)^{m+1}y(\tau)}^2 - \norm{B_d(\tau)^{m+1}}^2 \right) + o(\tau^{2k+1}) \\
    & = \mathcal O(\tau^{2m+1}) + o(\tau^{2m+1}) + o(\tau^{2k+1}) = o(\tau^{2k+1})
\end{align*}
by Remark \ref{myrem4.4}. This yields
\[ \norm{\sqrt{B_H} B_c^n y(\tau)} = o(\tau^{(k+1)-1-n}) \quad\text{for}\quad 0 \le n < k+1, \]
as desired.
\end{proof}

The following theorem yields the optimal upper bound for the asymptotic expansion of $\norm{ B_d(\tau)^{m+1}}$, the norm of the scaled Cayley transform, as a function of the step size and in the limit  $\tau\to 0+$. Thus it is the discrete analogue of Theorem~\ref{operator_norm_upper_bound}.

\begin{theorem}
\label{mythm4.7}
Let $B_c$ have hypocoercivity index $\mHC \in \N_0$. Then
\[ \norm{B_d(\tau)^{\mHC+1}}^2 \le 1 - \tau^{2\mHC+1} \frac{2}{\binom{2\mHC}{\mHC}} \lim_{\delta \to 0+} \inf_{\norm{\sqrt{B_H}B_c^p y} \le \delta, 0 \le p < \mHC, \norm{y}=1} \norm{\sqrt{B_H} B_c^{\mHC} y}^2 + o(\tau^{2\mHC+1}) \]
for $\tau\to0+$.

Recall from Theorem \ref{mythm4.3} that, for $\mHC=0$ the infimum is just taken over $\SS$.
\end{theorem}

\begin{proof}
    Choose some $y\colon\R^+ \to \HH$ satisfying $\Vert y(\tau) \Vert = 1$ for $\tau>0$, and
    $\norm{B_d(\tau)^{\mHC+1}y(\tau)} - \norm{B_d(\tau)^{\mHC+1}} = o(\tau^{2\mHC+1})$. Lemma \ref{mylem4.6} then implies that
\begin{equation}\label{mythm4.7-decay}
    \norm{\sqrt{B_H}B_c^n y(\tau)} = o(\tau^{\mHC-1-n})\quad \text{ for } 0 \le n < \mHC.
\end{equation}
Applying Lemma \ref{mylem4.2} and Lemma \ref{mylem4.5} (with $k$ and $d$ there, replaced by $\mHC$ here) we obtain
\begin{align*}
    \norm{B_d(\tau)^{\mHC+1}y(\tau)}^2 \le 1 - 2\tau \binom{2\mHC}{\mHC}^{-1} \tau^{2\mHC} \norm{\sqrt{B_H}B_c^{\mHC} y(\tau)}^2 + g(\tau),
\end{align*}
with $g(\tau) = o(\tau^{2\mHC+1})$.
Now, for any $\delta>0$ (due to \eqref{mythm4.7-decay}) we can find $T_\delta>0$ such that for any $0 < \tau \le T_\delta$ it holds that:
\[ \norm{B_d(\tau)^{\mHC+1}y(\tau)}^2 \le 1 - c_\delta\, \tau^{2\mHC+1}  + g(\tau), \]
with
\[
  c_\delta:= 2\binom{2\mHC}{\mHC}^{-1} \inf_{\substack{\norm{\sqrt{B_H}B_c^p y} \le \delta, 0 \le p < \mHC,\\ \norm{y}=1}} \norm{\sqrt{B_H}B_c^{\mHC} y}^2 .
\]
Note that $y(\tau)\big|_{(0,T_\delta]}$ is
included in the set in which we minimize. Moreover, the function $g$ is independent of $\delta$, $c_\delta$ increases as $\delta\to0+$, but $T_\delta$ (typically) decreases to 0 as $\delta\to0+$.

Using the fact that $\norm{B_d(\tau)^{\mHC+1}y(\tau)} - \norm{B_d(\tau)^{\mHC+1}} = o(\tau^{2\mHC+1})$, for every fixed $\delta>0$  we obtain that
\[
  \limsup_{\tau\to0+} \frac{\|B_d(\tau)^{\mHC+1}\|^2-1}{\tau^{2\mHC+1}} =
  \limsup_{\tau\to0+} \frac{\|B_d(\tau)^{\mHC+1}y(\tau)\|^2-1}{\tau^{2\mHC+1}} \le -c_\delta.
\]
Hence
\[
  \limsup_{\tau\to0+} \frac{\|B_d(\tau)^{\mHC+1}\|^2-1}{\tau^{2\mHC+1}}
\le -\lim_{\delta\to0+}c_\delta.
\]
which is a reformulation of the claim.
\end{proof}

The lower bound from Theorem \ref{mythm4.3} and the upper bound from Theorem \ref{mythm4.7} now imply the asymptotic expansion \eqref{eqqq1} and hence Theorem \ref{th:FinDiff-limit}.

%%%%%%%%%%%%%%%%%%%%%%%%%%%%%%%%%%%%%%%%%%%%%%%%

\section{Maximally coercive/contractive representations}\label{sec:optimal}
In this section we study maximally coercive/contractive representations of continuous-time linear systems of the form \eqref{linodecc} as (semi-)dissipative systems  and discrete-time linear systems of the form \eqref{linodedd}
as \emph{(semi-)contractive systems}.

\subsection{Continuous-time Lyapunov transformations}\label{sec:ctlyatf}
Even if the original system \eqref{linodecc} is \emph{(stable) asymptotically stable}, which implies that all eigenvalues of $-B_c$ have (nonpositive) negative real part, then the system is typically not (semi-dissipative) dissipative.
However, in this case one can, always in the matrix case and under some extra assumption in the operator case, perform a \emph{Lyapunov transformation} that makes the system semi-dissipative, see e.g. \cite{Adr95,HinP05}.

Consider  $B_c\in \mathbb C^{d,d}$, and determine a positive definite solution $X=X^*> 0$ of the \emph{Lyapunov inequality}
\begin{equation}\label{lyaineqcc}
B_c^* X+X B_c\, (\geq) >0.
\end{equation}
If such a solution exists, then  $ V(x):=\frac{1}{2} \Re(x^* Xx) $ is a Lyapunov function, and this implies that the solution $x$ is (stable) asymptotically stable, see e.g. \cite[\S 3.3.4]{HinP05}.
The solution $X=X^*>0$ then defines a  norm
\begin{equation}\label{ltvnormcc}
\| x \|_{X}:= (x^* X x)^{\frac 12}
\end{equation}
so that in this norm the solution $x(t)$ is (monotonically)
strictly monotonically decreasing in time.
Alternatively, we can multiply~\eqref{linodecc} with the positive definite square root $X^{\frac 12}$ and
after a  change of basis, $y=X^{\frac 12}x$ \cite{AchAM23ELA,BeaMV19} we obtain the ODE
\begin{equation}\label{dhode} \dot{y}= -\Big(X^{\frac 12}B_c X^{-\frac 12}\Big)y=: -(\widetilde {B_H}-\widetilde{B_S})y,
\end{equation}
where the (negative) right-hand side has Hermitian part
\[
\widetilde {B_H}:=\frac 12 (X^{\frac 12}B_c X^{-\frac 12} +(X^{\frac 12}B_c X^{-\frac 12})^*)=\frac 12 (X^{-\frac 12}(XB_c+B_c^*X)X^{-\frac 12}) > 0\  (\geq 0),
\]
respectively,
and skew-Hermitian part
\[
%J:=-\frac 12
-\widetilde {B_S}:=\frac 12
(X^{\frac 12}B_c X^{-\frac 12}-(X^{\frac 12}B_c X^{-\frac 12})^*).
\]

Note that the transformation to semi-dissipative form is not unique, so one might ask the question what is the relation between different (semi)-dissipative transformations that are obtained in this way.
Since for both transformation matrices in \eqref{dhode} we perform a similarity transform, the spectrum of the transformed system matrices is the same as that of the original system.

As one option to use the freedom in the solution $X$, in \cite{AAC16} Lemma 2, or \cite{AAS} Lemma 2.6 in the real case, it is suggested to solve the Lyapunov-like inequality
\begin{equation}\label{optx}
XB_c+B_c^* X \geq -2\mu X,
\end{equation}
where $\mu<0$ is the spectral abscissa of $-B_c$. It has been mentioned in \cite{Arnoldtalk} that this choice leads to a smallest eigenvalue of the Hermitian part $B_H$ that is equal to $-\mu$. In more detail one has the following result.
\begin{lemma}\label{lem:optimalX}
Let $-B_c\in \mathbb C^{d,d}$ have all eigenvalues in $\mathbb H$  and let $\mu<0$ be the spectral abscissa of $-B_c$.
\begin{enumerate}[label=(\roman*)]
\item If every eigenvalue with real part $\mu$ is semi-simple,  then there exists
$0<X\in \C^{d,d}$  such that
\[
XB_c+B_c^* X\geq -2\mu X,
\]
and the inequality is not strict, i.e., there exists a nonzero vector $z\in \mathbb C^d$ such that $z^*(XB_c+B_c^* X+2\mu X)z=0$.

\item If at least one eigenvalue $\lambda$ of $-B_c$ with real part $\mu$ is defective, i.e., there exists a Jordan block of size larger than one associated with $\lambda$,  then for every $\epsilon >0$ there exists $X>0$ such that
\[
XB_c+B_c^* X\geq -2(\mu+\epsilon) X.
\]
\end{enumerate}
\end{lemma}
\begin{proof}
(i)
Let $\lambda_1,\dots,\lambda_k$ be all eigenvalues of $-B_c$
with real part $\mu$, counted with algebraic multiplicity.
Since these eigenvalues are semi-simple, there exists an
invertible matrix $T\in \mathbb C^{d,d}$ such that
\[
-S=-T^{-1} B_c T = \diag(\lambda_1,...,\lambda_k, B_2),
\]
which can be obtained, e.g., from the Jordan form of $B_c$. Moreover, all eigenvalues of the matrix $B_2$ have real part smaller than $\mu$. Hence, $C \coleq B_2 - \mu I$ satisfies $\Lambda(C) \subseteq \HHH$. Thus, we define
\[ X_2 \coleq \int_0^\infty e^{C^*t}e^{Ct}\,dt > 0, \]
which is the unique solution of the continuous Lyapunov equation
\[ C^*X_2 + X_2C = - I, \]
and hence
\[ B_2^*X_2 + X_2B_2 = 2 \mu X_2 - I \le 2 \mu X_2.
\]
Then, defining $X:= T^{-*} \diag(1,...,1,X_2) T^{-1}$ yields
\begin{align*}
 B_c^* X+ XB_c &= T^{-*}S^*T^{*}X+X TS T^{-1}\\
&= T^{-*}(S^* T^{*} X T+T^{*} X T S) T^{-1} \\
&=  T^{-*} \diag\big(-2\mu,\dotsc,-2\mu, -B_2^* X_2-X_2 B_2\big) T^{-1},
\end{align*}
and thus the claim follows, since
\begin{align*}
B_c^* X+XB_c +2\mu X & =T^{-*} \diag\big(0,...,0, -B_2^* X_2-X_2 B_2+2\mu X_2 \big) T^{-1} \\
& = T^{-*} \diag\big(\underbrace{0,\dotsc,0}_k, \underbrace{1, \dotsc, 1}_{d-k} \big) T^{-1}\ge0.
\end{align*}
Since $k\ge1$, the nonzero vector $z=Te_1$ satisfies
\[
z^*(B_c^*X+XB_c+2\mu X)z=0.
\]

(ii) If an  eigenvalue $\lambda$ associated with the largest real part is defective,  then
\[
-S=-T^{-1} B_c T=
\begin{bmatrix}
     \lambda & b_2 \\ 0 & B_2
\end{bmatrix},
\]
with $b_2\neq 0$, where the norm of $b_2$ can be made arbitrarily small by a scaling of $b_2$ in the Jordan canonical form. Proceeding as in part (i), now with
$C:=B_2-(\mu+\epsilon)I$, it follows that
\[
B_c^* X+XB_c+2\mu X=T^{-*} \begin{bmatrix}
    0 & -b_2 \\ -b_2^* & -B_2^* X_2-X_2 B_2+2\mu X_2
\end{bmatrix}
T^{-1}
\]
is indefinite, while $B_c^* X+XB_c +2(\mu+\epsilon) X$ can be made definite by the scaling of $b_2$.
\end{proof}

Since  $B_c$ and $X^{\frac 12}B_c X^{-\frac 12}$ have the same spectral abscissa, Lemma~\ref{lem:optimalX} (i) shows (for the case that all eigenvalues of $-B_c$ with real part $\mu$ are semi-simple) that the  similarity transformation with $X^{\frac 12}$ from \eqref{optx} maps the %hypocoercive
system \eqref{linodecc} to a system of form \eqref{dhode} with \emph{maximal coercivity constant}, i.e., $\widetilde {B_H}$ has the maximal possible smallest eigenvalue. Furthermore, the field of values of the transformed system lies fully on the left of or on the line $\mu+i \mathbb R$. Such a solution  was already used %in the non-symmetric hypocoercivity method
for Fokker-Planck equations yielding sharp decay rates, see \cite[Lemma 2.11(i)]{AAS}.

\begin{remark}\label{rem:infdim}In the infinite dimensional case a similar result follows under weak conditions, see \cite{CurtainZwart}, Section 4, Theorem 4.1.3, as well as Exercise 4.17 and 4.18.
\end{remark}

\begin{remark}\label{rem_nonstrict}A result similar to Lemma~\ref{lem:optimalX} also holds if for the spectral abscissa one has $\mu=0$, i.e., if eigenvalues with real part zero occur.
The proof is analogous.
\end{remark}

\subsection{Discrete-time Lyapunov transformations}\label{sec:dtlyap}
For a linear time-invariant discrete-time system of the form \eqref{linodedd}
with $B_d\in \mathbb C^{d,d}$ we have an analogous result to Lemma~\ref{lem:optimalX}.
\begin{lemma}\label{lem:optimalXd}
Let $B_d\in \mathbb C^{d,d}$ have all eigenvalues in $\mathbb D$ and let $\rho<1$ be its spectral radius, i.e., the largest modulus of an eigenvalue of $B_d$.
\begin{enumerate}
    \item If every eigenvalue with modulus $\rho$ is semi-simple,  then there exists $X=X^*>0$  such that
\[
\rho^2 X-B_d^*XB_d\geq 0,
\]
and the inequality is not strict, i.e., there exists a nonzero vector $z\in \mathbb C^d$ such that $z^*(\rho^2 X-B_d^*XB_d)z=0$.

\item If at least one eigenvalue $\lambda$ with modulus $\rho$ is defective, i.e., there exist a Jordan block of size larger than one associated with $\lambda$,  then for every $\epsilon >0$ there exists $X=X^*>0$ such that
\[
(\rho^2+\epsilon) X-B_d^*XB_d\geq 0.
\]
\end{enumerate}
\end{lemma}
\begin{proof}
(i)
If $\rho=0$, then the semi-simplicity assumption implies $B_d=0$,
and the claim holds with $X=I$. Henceforth assume $\rho>0$.

Let $\lambda_1,\dots,\lambda_k$ be all eigenvalues of $B_d$ with modulus $\rho$, counted with algebraic multiplicity.
Since these eigenvalues are semi-simple, there exists an
invertible matrix $T\in \mathbb C^{d,d}$ such that
\[
S=T^{-1} B_d T = \diag(\lambda_1,...,\lambda_k, B_2),
\]
which can be obtained, e.g., from the Jordan form of $B_d$. Moreover, all eigenvalues of the matrix $B_2$ have modulus smaller than $\rho$. If $k=d$, then the block $B_2$ is absent, and choosing
$X:=T^{-*}T^{-1}>0$ gives $B_d^*XB_d=\rho^2X$,
which proves the claim. Henceforth assume $k<d$.
 Since  $\rho>0$, we can form $C \coleq B_2 / \rho$ which satisfies $\Lambda(C) \subseteq \mathbb D$. Then
\[
X_2 \coleq \sum_{l=0}^\infty (C^*)^l C^l > 0,
\]
is the unique solution of the discrete Lyapunov equation
\[
C^*X_2 C - X_2 = - I,
\]
and hence
\[ B_2^*X_2 B_2 - \rho^2 X_2
= \rho^2(C^* X_2 C -X_2) =-\rho^2 I\le0.
\]

Then, defining $X:= T^{-*} \diag(1,...,1,X_2) T^{-1}$
yields the claim, since
\begin{align*}
 B_d^* XB_d - \rho^2 X
&= T^{-*} \big(S^* \diag(1,...,1,X_2) S-\rho^2 \diag(1,...,1,X_2)\big) T^{-1} \\
&=  T^{-*} \big(\diag(0,\dotsc,0, B_2^* X_2B_2-\rho^2 X_2) \big) T^{-1} \\
&=  T^{-*} \big(\diag(\underbrace{0,\dotsc,0}_k, \underbrace{-\rho^2, \dotsc, -\rho^2}_{d-k}) \big) T^{-1} \le0 .
\end{align*}
Finally, since $k\ge1$, the nonzero vector $z=Te_1$ satisfies
\[
z^*(B_d^*X B_d-\rho^2X)z=0.
\]

(ii) If an  eigenvalue $\lambda$ with modulus $\rho$
is defective,  then
\[
S=T^{-1} B_d T=
\begin{bmatrix}
     \lambda & b_2 \\ 0 & B_2
\end{bmatrix},
\]
with $b_2\neq 0$, where the norm of $b_2$ can be made arbitrarily small by a scaling of $b_2$ in the Jordan canonical form. Proceeding as in part (i), now with
$C:=B_2/\sqrt{\rho^2+\epsilon}$, it follows that

\[
\rho^2X-B_d^*XB_d=T^{-*} \begin{bmatrix}
  0 & -\bar\lambda b_2 \\- \lambda b_2^*  & \rho^2X_2-B_2^* X_2B_2- b_2^*b_2  %X_2 B_2X_2-b_2^*b_2
\end{bmatrix}
T^{-1}
\]
is not positive semidefinite, while $(\rho^2+\epsilon)X-B_d^*XB_d$ can be made positive definite by the scaling of $b_2$.
\end{proof}

Note that the construction in Lemma~\ref{lem:optimalXd} (i) then leads to a representation that is maximally contractive (when all eigenvalues with modulus $\rho$ are semi-simple).

\begin{remark}\label{rem_nonstrict-discrete} A result similar to Lemma~\ref{lem:optimalXd} also holds if for the spectral radius one has $\rho=1$, i.e., if eigenvalues on the unit circle occur.
The proof is analogous.
\end{remark}

Discretization of \eqref{linodecc} with the implicit midpoint rule yields a discrete-time
system with $B_d=(I+\frac{\tau}{2}B_c)^{-1}(I-\frac{\tau}{2}B_c)$ and the discrete-time Lyapunov inequality holds for $B_d$ if and only if the continuous-time inequality holds for $-B_c$.
Therefore  we can choose the formulations  in both cases accordingly.

\subsection{The effect of a Lyapunov transformation on the numerical solution}

In this subsection we study the effect of the Lyapunov transformation to semi-dissipative form in \eqref{dhode} on the numerical solution of \eqref{linodecc}.

Consider the well-known error estimates, see e.g. \cite{HaiNW93}, for the solution of initial value problems for ordinary differential equations $\dot x=f(x)$, $x(t_0)= x^0$ with a one-step method for the approximations $u_i\approx x(t_i)$ on a grid $t_0<t_1<\cdots<t_N$ (with $\tau_i=t_{i+1}-t_{i}$) given by
\begin{equation}\label{ostepmethod}
    u_{i+1}=u_i+\tau_i \Phi(t_i, u_i, \tau_i),\ u_0\approx x^0,
\end{equation}
with the increment function $\Phi$ satisfying a Lipschitz condition
\begin{equation} \label{lipcond}\| \Phi(t,x_1,\tau)-\Phi(t,x_2,\tau) \| \leq L_x  \|x_1-x_2 \|,
 \end{equation}
 for a norm $\|\cdot \|$.
 Then
\begin{equation}\label{errest}
 \|x(t_i)-u_i \|\leq \big(\|x(t_0)-u_0\|+(t_i-t_0)\theta_{\max,x}\big) e^{L_x(t_i-t_0)},
 \end{equation}
 where
 \[
\theta_{\max,x} =\max_{t\in[t_0,t_N]} |\theta(t)|
\]
is the maximal local discretization error of the method in $x$ coordinates. Define $\theta_{\max,y}$ analogously when discretizing the system in $y$ coordinates. It follows that
\[
\theta_{\max,y} \leq \theta_{\max,x}
\|X^{\frac 12}\|.
\]

Let us consider the original $L_2$ norm
$\|x(t)\|$ and  the same method applied to the transformed system  with $y(t)=X^{\frac 12}x(t)$ with approximations $v_i=X^{\frac 12}u_i~\approx y(t_i)$ and Lipschitz constant $L_y$.
Then
\begin{equation}
\|x(t_i) -u_i\| \leq \| X^{-\frac 12}\|\,  \|y(t_i)-v_i \|\leq \| X^{-\frac 12}\|\, \big(\|y(t_0)-v_0\|+(t_i-t_0) \theta_{\max,y}\big) e^{L_y(t_i-t_0)}.\label{erresty}
 \end{equation}
 If we consider the implicit midpoint rule (with constant stepsize $\tau$) applied to \eqref{linodecc}, then
\[ \Phi(x,\tau)= -(I+\frac {\tau}{2} B_c)^{-1}B_cx
\]
so that
\[L_y =\|-(I+\frac{\tau}{2}X^{\frac 12}B_cX^{-\frac 12})^{-1} X^{\frac 12}B_cX^{-\frac 12} \|
\leq \| X^{\frac 12}\|\, \|X^{-\frac 12}\|L_x
\]
and hence
\[
e^{L_y(t_i-t_0)}\leq e^{L_x(t_i-t_0)\|X^{\frac 12}\|\, \|X^{-\frac 12}\|}.
\]
Inserting this into \eqref{erresty} %{errest}
we obtain the estimate
\[
\|x(t_i) -u_i\| \leq \| X^{\frac 12}\|\, \|X^{-\frac 12}\|\,
\big(\|x(t_0)-u_0\|+(t_i-t_0) \theta_{\max,x}\big) e^{L_x(t_i-t_0)\| X^{\frac 12}\|\, \|X^{-\frac 12}\|}.\]

This shows that the transformation to semi-dissipative form leads to a multiplication of the error estimate by the condition number of $X^{\frac 12}$ and by an increase of the exponent by the same factor.

Note that the same estimate holds for all Runge-Kutta methods when applied to linear systems, since usually the increment function of the transformed system is similar to that of the original system and therefore the Lipschitz constant is just scaled with the condition number of the transformation matrix.

\section*{Conclusion}
New proofs for the characterizations of the short-time decay of the solutions
of linear continuous-time and discrete-time evolution equations from the initial value are presented for both systems. It is shown that essentially the same proof
technique can be employed in both systems. The constants in the leading terms of the %solution
propagator norm expansion are determined. When the
discrete-time system arises from the implicit midpoint discretization (scaled Cayley-transform) of
a continuous-time system, it is shown that the norm of the continuous-time solution operator is
approximated to a higher order than expected from the order of the discretization.

Since the representation of a linear system as a semi-dissipative or semi-contractive
system is not unique, but can be modified by a change of basis, we discuss the construction of
maximally coercive/contractive representations of hypocoercive and hypocontractive systems and the effect of different
representations on the error estimates for the numerical solution. We note that the extension of the results in this paper to unbounded operators is an interesting research topic, see \cite{GuoZ06}.

\bigskip
\textbf{Acknowledgments. } This research was funded in part by the Austrian Science Fund (FWF) project 10.55776/F65. For open-access purposes, the authors have applied a CC BY public copyright
license to any author-accepted manuscript version arising from this
submission. One of the authors (SE) was also supported by the Vienna School of Mathematics (VSM).

\bibliographystyle{abbrv}
\bibliography{references}
\end{document}